\documentclass[12pt]{amsart}
\usepackage[utf8]{inputenc}
\usepackage{a4wide}
\usepackage{amsmath}
\usepackage{amssymb}
\usepackage{amsfonts}
\usepackage{amsthm}
\usepackage{stix}
\usepackage{mathrsfs}
\usepackage{enumitem}
\usepackage[inline]{showlabels}

\usepackage{soul}

\usepackage{caption}
\usepackage{subcaption}

\usepackage{bbm}
\usepackage{import}
\usepackage{listings}
\usepackage{xcolor}
\usepackage{algcompatible}
\usepackage{algorithm}
\usepackage[noend]{algpseudocode}
\floatname{algorithm}{Algorithm}

\usepackage{etoolbox}\AtBeginEnvironment{algorithmic}{\small} 
\usepackage{float}
\newfloat{algorithm}{H}{lop}
\usepackage{comment}
\usepackage{biblatex}
\usepackage{graphicx}
\usepackage{tikz}
\usepackage{subcaption}
\usepackage{hyperref}
\usepackage{accents}
\newcounter{dummy}
\makeatletter
\newcommand\myitem[1][]{\item[#1]\refstepcounter{dummy}\def\@currentlabel{#1}}
\makeatother

\makeatletter
\def\BState{\State\hskip-\ALG@thistlm}
\makeatother

\newtheorem{theorem}{Theorem}[section]
\newtheorem{proposition}[theorem]{Proposition}
\newtheorem{lemma}[theorem]{Lemma}
\theoremstyle{definition}
\newtheorem{definition}[theorem]{Definition}
\newtheorem{assumption}[theorem]{Assumption}
\newtheorem{problem}[theorem]{Problem}

\newtheorem{property}[theorem]{Property}

\theoremstyle{remark}
\newtheorem{remark}[theorem]{Remark}
\numberwithin{equation}{section}

\newcommand{\eps}{\varepsilon}

\newcommand{\p}{\mathbb{P}}
\newcommand{\E}{\mathbb{E}}
\newcommand{\R}{\mathbb{R}}
\newcommand{\N}{\mathbb{N}}

\newcommand{\calP}{\mathcal{P}}
\newcommand{\calM}{\mathcal{M}}
\newcommand{\calD}{\mathcal{D}}

\newcommand{\calL}{\mathcal{L}}
\newcommand{\calB}{\mathcal{B}}

\newcommand{\calV}{\mathcal{V}}

\newcommand{\err}{J}

\newcommand{\conv}{\text{Conv}}

\newcommand{\FL}{\mathsf{FL}}

\newcommand\norm[1]{\left\lVert#1\right\rVert}
\newcommand{\norml}[1]{{\left\vert\kern-0.25ex\left\vert\kern-0.25ex\left\vert #1\right\vert\kern-0.25ex\right\vert\kern-0.25ex\right\vert}_{\lambda}}
\newcommand{\mfd}{\mathfrak{d}}
\newcommand{\mfL}{\mathscr{L}}
\newcommand{\mfS}{\mathfrak{S}}

\newcommand{\sfb}{\mathsf{b}}
\newcommand{\sfP}{\mathsf{P}}
\newcommand{\sfN}{\mathsf{N}}
\newcommand{\scrG}{\mathscr{G}}
\newcommand{\Pinf}{\sfP_{\!\!\infty}}

\newcommand{\interior}[1]{%
  {\text{int}(#1)}%
}

\def\Xint#1{\mathchoice
{\XXint\displaystyle\textstyle{#1}}%
{\XXint\textstyle\scriptstyle{#1}}%
{\XXint\scriptstyle\scriptscriptstyle{#1}}%
{\XXint\scriptscriptstyle\scriptscriptstyle{#1}}%
\!\int}
\def\XXint#1#2#3{{\setbox0=\hbox{$#1{#2#3}{\int}$ }
\vcenter{\hbox{$#2#3$ }}\kern-.55\wd0}}

\def\dashint{\Xint-}

\def\Xintfrac#1{\mathchoice
{\XXintfrac\displaystyle\textstyle{#1}}%
{\XXintfrac\textstyle\scriptstyle{#1}}%
{\XXintfrac\scriptstyle\scriptscriptstyle{#1}}%
{\XXintfrac\scriptscriptstyle\scriptscriptstyle{#1}}%
\!\int}
\def\XXintfrac#1#2#3{{\setbox0=\hbox{$#1{#2#3}{\int}$ }
\vcenter{\hbox{$#2#3$ }}\kern-.6\wd0}}

\def\dashintfrac{\Xintfrac-}
\usepackage{soul}

\DeclareMathOperator*{\argmin}{arg\,min}
    \DeclareMathOperator*{\argmax}{arg\,max}

\title[Stochastic convergence of a class of greedy-type algorithms]{Stochastic convergence of a class of greedy-type algorithms for Configuration Optimization Problems}
\author{E. Nielen, O. Tse}
\date{\today}

\begin{document}
\begin{abstract}
    Greedy Sampling Methods (GSMs) are widely used to construct approximate solutions of Configuration Optimization Problems (COPs), where a loss functional is minimized over finite configurations of points in a compact domain. While effective in practice, deterministic convergence analyses of greedy-type algorithms are often restrictive and difficult to verify.

    We propose a stochastic framework in which greedy-type methods are formulated as continuous-time Markov processes on the space of configurations. This viewpoint enables convergence analysis in expectation and in probability under mild structural assumptions on the error functional and the transition kernel. For global error functionals, we derive explicit convergence rates, including logarithmic, polynomial, and exponential decay, depending on an abstract improvement condition.

    As a pedagogical example, we study stochastic greedy sampling for one-dimensional piecewise linear interpolation and prove exponential convergence of the $L^1$-interpolation error for $C^2$-functions. Motivated by this analysis, we introduce the Randomized Polytope Division Method (R-PDM), a randomized variant of the classical Polytope Division Method, and demonstrate its effectiveness and variance reduction in numerical experiments.
\end{abstract}

\maketitle
% \tableofcontents

\section{Introduction}
\label{sec: introduction}
Greedy Sampling Methods (GSMs) are used in many applications. These applications include function approximations (\cite{barron2008approximation, devore1996some}), reduced basis methods (\cite{chen2010certified, deparis2008reduced, huynh2007reduced, negri2013reduced, ngoc2005certified, rozza2011reduced, rozza2008reduced, rozza2007stability, rozza2009reduced, sen2008reduced, veroy2003posteriori}), interpolation (\cite{barrault2004, campagna2021greedy, chen2014weighted, maday2007general, schaback2014greedy}), and others (\cite{mirzasoleiman2015lazier, schaback2014greedy, wenzel2024data}). Many of these problems can be reformulated as Configuration Optimization Problems (COPs), where GSMs can be used to approximate solutions. Convergence guarantees for greedy methods can be hard to obtain in a deterministic setting and can be restrictive. In this work, we formulate Stochastic Greedy Sampling Methods and derive convergence results with high probability.

\subsection*{Configuration Optimization Problems}
GSMs are often applied to approximate solutions to COPs. In COPs, the objective is to minimize a loss function as a function of configurations of points within a given compact set $P\subset\R^d$. In this setting, a configuration $\eta$ with $n \in \N$ points in $P$ is an element of
\[
\Omega_n \coloneq \left\{\eta = (p_1, \ldots, p_n, \phi, \phi, \ldots)\, | \, p_i \in P \right\} \subset P^{\N}
\]
% \[
%     \Omega \coloneq \bigl\{(\zeta_i)_{i\in\N}: \zeta_i\in  (P \cup \{\phi\})\;\;\text{for all $i\in\N$}\bigr\},
% \]
where $\phi$ is a so-called graveyard state.
The space of sequences is then given by the disjoint union
\[
    \Omega \coloneq \bigsqcup_{n \ge 0} \Omega_n.
\]
For each $n\in\N$, we further define the map 
\[
    \Omega_n\ni\eta\mapsto \Lambda_n(\eta) = \{p_1,\ldots,p_n\} \in \varGamma(P),\qquad \eta = (p_1, \ldots, p_n, \phi, \phi, \ldots),
\]
where $\varGamma(P) = \{ A\subset P: \#A<\infty\}$ denotes the family of finite subsets of $P$.
% In practice, we can only consider configurations $\eta\in\Omega$ with finitely many points in $P$, i.e.,
% \[
%     \eta\in \Omega_n \coloneq \bigl\{\zeta \in \Omega : \#\{i\in\N:\zeta_i \neq \phi\} = n\bigr\}.
% \]

\medskip
The Configuration Optimization Problem then reads:
\begin{problem}
\label{problem: COP}
    Let $P \subset \R^d$ be a compact set and $\ell{:} \varGamma(P) \to [0, +\infty)$ be a given loss function. For a fixed $n \in \N$, we aim to find
    \begin{align}\label{problem:cop}\tag{\sf COP}
        \gamma \in \argmin_{\eta \in \Omega_n} \mfL(\eta),\qquad \mfL = \ell\circ\Lambda.
    \end{align}
\end{problem}
% By construction, the loss function $\mfL$ is invariant under permutations of sequences $\eta \in \Omega$ and the graveyard states do not influence its value. In other words, the ordering of the points in $\eta$ is irrelevant for the value of the loss function and 

% . The graveyard states do not influence the value of the loss function, but allow

% us to express any set of points $\{p_1, \ldots, p_n\} \subset P$ as an infinite sequence $\eta = (p_1, \ldots, p_n, \phi, \ldots) \in \Omega$. We note that this assignment is non-unique.
\begin{remark}
    We note that the definition of a COP differs from \cite{nielen2025polytope}. The setting presented here allows us to keep track of the order in which points are added, making the stochastic process considered below Markovian.
\end{remark}

The initial motivation for our study stems from the Reduced Basis Method (RBM) in the context of Model Order Reduction. In RBM, the aim is to approximate a solution manifold $\calM = \{u(p) \in \calV : p \in P\}$, where $u(p)$ is the solution of a PDE governed by the parameter $p$ in some Hilbert space $\calV$. Given a configuration $\eta\in\Omega_n$, a reduced basis is the set $\{u(p):p\in \Lambda_n(\eta)\}$, whose span $V_\eta$ is a linear space approximating the solution manifold $\calM$. The loss function could be given by
\begin{align*}
    \mfL(\eta) = \max_{q \in P} \norm{u(q) - \mathrm{Proj}_{V_\eta} u(q)}_\calV^2,
\end{align*}
where $\mathrm{Proj}_{V_\eta}$ is a projection operator onto the linear space $V_\eta$. Other examples can be found in the context of the Empirical Interpolation Method \cite{barrault2004}, Optimal Experimental Design \cite{ucinski2004optimal}, and active learning for regression \cite{wu2019active}. 

Since it is generally infeasible to find an exact solution to COPs, \emph{greedy methods} are employed to approximate a solution. These methods iteratively construct the configuration $\eta \in \Omega$. Generally, the methods is initiated with $\eta^1 = (p_1, \phi, \ldots) \in \Omega$, where $p_1 \in P$ arbitrarily chosen. In the next steps, the configuration $\eta^j$ is updated by selecting a new point $p_{j+1} \in P$ and setting $\eta^{j+1} = (p_1, \ldots, p_j, p_{j+1}, \phi, \ldots)$. The selection criteria of the point $p_{j+1}$ depends on the specific greedy method. Classically, for loss functions of the form $\mfL(\eta) = \max_{p \in P} \err(p, \eta)$, for some given error function $\err{:} P \times \Omega \to [0,+\infty)$, the idea is to select $p_{j+1}$ satisfying
\begin{align}
\label{eq: p update}
    p_{j+1} \in \argmax_{q \in P} \err(q, \eta^{j-1}).
\end{align}
In other words, we select the point in $P$ with the highest error value with the hope that this point estimates the best possible update. Since it is often infeasible to compute~\eqref{eq: p update} exactly, this greedy strategy is often replaced by 
\begin{align}
\label{eq: weak greedy}
    p_{j+1} \in \argmax_{q \in S} \err(q, \eta^{j-1}),
\end{align}
where $S \subset P$ is a discrete sample set. We refer to this strategy as {\it weak greedy sampling}.

\subsection*{Literary overview}
Other greedy strategies exist. In \cite{urban2014greedy}, the authors perform gradient descent for several starting points to approximate the global maximum argument~\eqref{eq: p update}. Many alternative methods revolve around the sample set in weak greedy sampling. The quality of the weak greedy strategy depends on the representativeness of the parameter set $P$ by the sample set $S$. In practical implementations, the sample size $|S|$ of the sample set $S$ suffers from the curse of dimensionality. Therefore, many alternative methods exist to overcome these scalability issues, often considering different sample sets $S^j$ after each update step. Examples include \cite{hesthaven2014efficient}, where irrelevant samples are removed from $S$, and new, possibly relevant, samples are added. In \cite{haasdonk2011training}, the sample set is also adaptively enriched based on the error within a validation set. In \cite{sen2008reduced}, the weak greedy algorithm is performed on several smaller, disjoint training sets, and in \cite{jiang2017offline}, the authors use the successive maximization method to construct a surrogate training set. In \cite{nielen2025polytope}, the Polytope Division Method (PDM) is introduced. In PDM, the parameter set $P$ is divided into polytopes, and the sample set consists of the barycenters of these polytopes. In this paper, we introduce a randomized version of this algorithm (R-PDM).

In the context of Reduced Basis Methods, the convergence of greedy methods is often investigated based on the Kolmogorov $n$-width \cite{binev2011convergence, cohen2020reduced}. The Kolmogorov $n$-width is defined as the $n$-dimensional linear space that minimizes the approximation error of solution manifold $\calM$ in the supremum norm. In \cite{binev2011convergence}, the authors show that the polynomial or exponential decay of the Kolmogorov $n$-width implies polynomial or exponential decay of the weak greedy algorithm, respectively, if it can be guaranteed that the error of the selected parameter is at least a factor $\gamma \in (0,1]$ of the exact maximum error. This maximum error is typically unknown, leading to an assumption that is hard to guarantee. Therefore, in \cite{cohen2020reduced}, they investigate the required sample size $|S|$ to guarantee this assumption is satisfied with high probability only for the cases where the solution map is analytic in the parameters. Not only can these assumptions be hard to guarantee, but the comparison to the Kolmogorov $n$-width can also be restrictive. Outliers in the solution manifold can dominate the convergence rates, and might present an overly cautious perspective.

A different view is presented in \cite{li2025new}, where the convergence estimates are compared with the metric entropy numbers. The metric entropy number represents the smallest radius necessary to cover a compact set with $2^n$ balls of this radius. This perspective leads to sharper bounds than the classical comparison to the Kolmogorov $n$-width, but still depends on the same assumptions. A different comparison is investigated in this paper, where we model greedy methods as stochastic processes and derive probabilistic convergence results.

\subsection*{Outline of paper}
% Using the corresponding forward Kolmogorov equation, we can relate the generator of the process to a probability measure $P_t \in \calP(\Omega)$ where $P_t = \rm{Law}(\eta_t)$ and $\eta_t$ denotes the configuration at time $t$. 
Stochastic greedy methods form a broader class than deterministic greedy methods because deterministic methods can be recovered by setting the kernel $\lambda$ as Dirac measures. The benefits of re-framing greedy methods as stochastic processes are three-fold: 
\begin{enumerate}
    \item Convergence statements in probability and expectation are less restrictive than convergence statements in maximum error.
    \item Stochastic methods can be used to prove convergence in probability and expectation.
    \item The viewpoint of the stochastic process enables a broader view of greedy methods, and in particular, the introduction of the kernel $\lambda$ can lead to new greedy-type algorithms.
\end{enumerate}

To model greedy methods as stochastic processes, we consider a configuration $\eta$ that changes over time---a configuration $\eta = (\eta_1, \ldots, \eta_n, \phi, \ldots)$ transitions to the configuration 
\[
    \eta {\oplus} y = (y,\eta_1, \ldots, \eta_n, \phi, \ldots)\quad\text{with rate $\lambda(\eta, dy)$.}
\]
In Section~\ref{sec: stochastic process}, we detail the precise definition of this $\oplus$-operator and the generator of this process.

In Section~\ref{sec: convergence}, we show convergence of functions $\err{:} P \times \Omega \to [0,+\infty)$ without rates under a minimal set of assumptions (cf.\ Assumption~\ref{assumption: J}) and for a broad class of kernels $\lambda$. Since these functions depend also on points in $P$, we call them \emph{local} functions.

For \emph{global} functions $\mathscr{G}{:} \Omega \to [0,+\infty)$ such as the loss function $\mfL$ in \eqref{problem:cop}, we show logarithmic, polynomial, and even exponential convergence rates (cf.\ Theorem~\ref{thm: convergence rate average}) under more stringent assumptions on $\scrG$ and the kernel $\lambda$ (cf.\ Assumption~\ref{assumption: average error}) in Section~\ref{sec: average improvement}

In Section~\ref{sec: interpolation}, we consider piecewise linear interpolation as a pedagogical example of a COP. We show that the $L^1$ interpolation error of general $C^2$-functions converges exponentially using stochastic greedy methods with a transition kernel satisfying Property~\ref{assumption: interpolation transition kernel}.

In Section~\ref{sec: rpdm}, we introduce the Randomized Polytope Division Method (R-PDM) and numerically show convergence results for the three pedagogical examples. We further numerically investigate the reduction in variance of the R-PDM over the uniform kernel, highlighting an additional benefit of R-PDM. 
% Finally, we provide a short summary and outlook in Section~\ref{sec: outlook}.

% show variance reduction in three different interpolation examples.

% Moreover, we formulate the Randomized Polytope Division Method (R-PDM) related to the original PDM. In \cite{nielen2024polytope}, the convergence of the error using PDM was only established in numerical examples, but the results in this paper prove the convergence in high probability for the randomized version. We show for a small interpolation example that the variance of R-PDM is lower than a stochastic process with uniform rates. 

The main contributions of this paper are the following:
\begin{enumerate}
    \item We formulate greedy methods as stochastic processes.
    \item We derive convergence results with and without rates for these methods under a varied set of assumptions.
    \item We introduce the R-PDM and provide analytical and numerical studies of non-trivial pedagogical examples.
\end{enumerate}

% \subsection*{Outline of paper}
% In Section~\ref{sec: stochastic process}, we write greedy-type algorithms as stochastic processes. We state the assumptions and write down the generator of the process. In Section~\ref{sec: forward Kolmogorov}, we connect the generator with a probability measure using the forward Kolmogorov equation. In Section~\ref{sec: convergence}, we show convergence of the error function under the simplest assumption. Under an alternative assumption, we prove convergence rates in  Section~\ref{sec: average improvement}. 

\section{Configuration construction as a stochastic process}
\label{sec: stochastic process}
In this section, we model greedy-type algorithms as a stochastic process and prove its well-posedness under certain assumptions. The stochastic process models the selection of particles of a configuration, i.e., its state space is $\Omega \coloneq \bigsqcup_{n \ge 0} \Omega_n$, where 
\[
\Omega_n \coloneq \left\{\eta = (\eta_1, \ldots, \eta_n, \phi, \ldots)\, | \, \eta_i \in P \right\},
\]
and $\phi$ denotes a graveyard state.
% The state space of this process is given by $\Omega \coloneq (f \colon \N \to (P \cup \{\phi\}))$, where $\phi$ is a so-called graveyard state.
% We describe configurations as an infinite sequence of points $\eta = (p_1, \ldots, p_n, \phi, \phi, \ldots)$, where $p_i \in P$, for all $i \in \{1, \ldots, n\}$ and $\phi$ is the graveyard state. 
% After updating the configuration, $\eta$ transitions to a new state $\sigma \in\Omega$.
% 
% We define the $\sigma$-algebra on $\Omega$ as $\Sigma \coloneq \sigma(\{\sfN^{-1}(B): B \in \calB_{\R}\})$, where $\sfN^{-1}$ denotes the pre-image of $\sfN$ and $\calB_{\R}$ denotes the Borel $\sigma$-algebra on $\R$. \oliver{I'm not sure we need the $\sigma$-algebra generated by $\sfN$. The Borel $\sigma$-algebra associated to $\mfd$ should be enough.}
Moreover, we consider the following metric on $\Omega$
    \begin{definition}
    \label{def: metric}
        Let $\eta, \sigma \in \Omega$. A metric $\mfd$ on $\Omega$ is defined as
        \begin{align*}
            \mfd(\eta, \sigma) = \sum_{i \in \N} \frac{1}{2^i} \bar{\mfd}(\eta_i, \sigma_i),
        \end{align*}
        where
        \begin{align*}
        \begin{cases}
            \;\bar{\mfd}(p, q) \coloneq |p - q|_2 \quad &\text{for} \, p, q \in P,\\
            \;\bar{\mfd}(p, q) \coloneq {\rm{diam}}(P) \quad &\text{for} \, p \in P, q = \phi, \;\text{or}\; p=\phi, q\in P,\\
            \;\bar{\mfd}(p, q) \coloneq 0 \quad &\text{for} \, p = q = \sigma.  
        \end{cases}
        \end{align*}
    \end{definition}
    We then equip $\Omega$ with the Borel $\sigma$-algebra $\calB_\Omega$ induced by the metric $\mfd$.

    % define the $\sigma$-algebra $\calB_\Omega$ on $\Omega$ as the Borel $\sigma$-algebra associated to $\mfd$, i.e. let $B_\mfd(\eta, r)$ denote a ball around $\eta \in \Omega$ with radius $r > 0$ associated to $\mfd$, then $\calB_\Omega \coloneq \sigma\big(\{B_\mfd(\eta, r) : \eta \in \Omega, r > 0\}\big)$.
% We define the $\sigma$-algebra on $\Omega$ as $\Sigma \coloneq \sigma(\{\sfN^{-1}(B): B \in \calB_{\R}\})$, where $\sfN^{-1}$ denotes the pre-image of $\sfN$ and $\calB_{\R}$ denotes the Borel $\sigma$-algebra on $\R$. \oliver{I'm not sure we need the $\sigma$-algebra generated by $\sfN$. The Borel $\sigma$-algebra associated to $\mfd$ should be enough.}

    \medskip
    As we mentioned ealier, a configuration $\eta$ transitions to a different configuration $\sigma$ by shifting its elements and adding a new point, expressed in terms of the $\oplus$-operator. 

    \begin{definition}
    The operator $\oplus{:} \Omega \times P \to \Omega$ is defined as
    \[
        \eta {\oplus} y \coloneq (y, \eta_1, \ldots, \eta_n, \phi, \ldots),\qquad \eta \in \Omega_n,\; y \in P.
    \]
    Moreover, we define the counting function $\sfN{:} \Omega \to \N$ as
\[
    \sfN(\eta) = n,\qquad \eta\in\Omega_n,
\]
counting the number of particles in a configuration $\eta$ that are elements of $P$.
    \end{definition}
    
    \begin{lemma}
        \label{lemma: compactness omega}
        The metric space $(\Omega,\mfd)$ is compact. Moreover, the maps $\oplus$ and $\sfN$ are continuous and, therefore, Borel measurable.
    \end{lemma}
    The proof of this lemma can be found in Appendix~\ref{app: compactness}.

    \medskip
    To characterize the possible transition over time, we now introduce the generator of the process. Let $B_b(\Omega)$ be the set of bounded Borel functions on $\Omega$. The generator $L{:} B_b(\Omega) \to B_b(\Omega)$ of the process is then given by
    \begin{align}
    \label{eq: generator}
    LF(\eta) \coloneq \int_{P} \bigl[F(\eta {\oplus} y) - F(\eta)\bigr] \lambda(\eta, dy).
    \end{align}
    % \[
    %     \eta \mapsto \eta \oplus y \quad\text{with rate $\lambda(\eta,dy)$}
    % \]
    Here, $F \in B_b(\Omega)$ is an observable and $\lambda$ is the transition kernel. In the context of Greedy Sampling Methods, an example of $F(\eta)$ could be the error function $\err(q, \eta)$ at some fixed point $q \in P$.
    Given this generator, we assume that the sequence $\eta$ transitions to $\eta {\oplus} y$ at the rate $\lambda(\eta, dy)$.
    A simple example of $\lambda(\eta, dy)$ is a uniform measure over $P$. 
    
    In Section~\ref{sec: rpdm}, we introduce the Randomized Polytope Division Method (R-PDM) and construct a transition kernel $\lambda$ such that the process corresponds to the construction of configurations based on R-PDM. We always assume an initial condition of the form $\eta_0 = (\eta_0, \phi, \ldots)$ for some $\eta_0 \in P$.

\medskip
Throughout, we assume that the transition kernel $\lambda$ satisfies 
\begin{assumption}
\label{assumption: rate function}
The transition kernel $\lambda{:}\Omega\times \calB_{\R^d}\to [0,+\infty)$ satisfies
\begin{enumerate}
    \item $\lambda(\eta,\cdot)\in\calP(P)$,
    \item for any $A\in\calB_{\R^d}$, the map $\Omega\ni \eta \mapsto \lambda(\eta,A)$ is Borel measurable,
    \item $\lambda(\eta, P) = 1$ for every $\eta\in\Omega$.
\end{enumerate}
    
% Moreover,
% \begin{align*}
%     \lambda(\eta, A) \ge \calL(A) \ge 0, \quad \text{for all}\, A \in \calB(P).
% \end{align*}
\end{assumption}

\begin{remark}
    We note that the results here may be generalized to the case were the transition kernel satisfies $\sup_{\eta\in\Omega} \lambda(\eta, P) <+\infty$ without any difficulties.
\end{remark}

Under Assumption~\ref{assumption: rate function} we have the following existence result, which follows from \cite[\S4.2]{ethier2009markov} upon showing that $\lambda$ gives rise to is a well-defined transition kernel $\kappa{:}\Omega\times \calB_\Omega\to [0,+\infty)$. For completeness, the proof of this statement is found in Appendix~\ref{app: well-definiteness}.

\begin{proposition}
    Let $\lambda$ be a transition kernel satisfying Assumption~\ref{assumption: rate function}. Then there exists a unique $\Omega$-valued Markov process $(\eta_t)_{t\ge 0}$ with bounded generator $L{:}B_b(\Omega)\to B_b(\Omega)$ defined in \eqref{eq: generator}.
\end{proposition}

% Now, we provide a link between the generator $L$ and the law of the process.
% Before we can conclude the existence of a probability measure $\sfP_t \in \calP(\Omega)$ satisfying the forward Kolmogorov equation, we must show that the generator~\eqref{eq: generator} indeed defines a process. This happens in Appendix~\ref{app: well-definiteness}.
Since $L$ is a generator of the process $(\eta_t)_{t\ge 0}$, the time marginal law $\sfP_t=\text{Law}(\eta_t) \in \calP(\Omega)$  satisfies the forward Kolmogorov equation
\begin{align}
\label{eq: forward kolmogorov}\tag{\textsf{FKE}}
\langle F, \sfP_t\rangle - \langle F, \sfP_s \rangle = \int_s^t \langle LF, \sfP_r\rangle \,dr, \quad \text{for all}\, F \in B_b(\Omega),
\end{align}
where $\langle F, \sfP\rangle \coloneq \int_\Omega F(\eta) \sfP(d\eta)$.
% We note that $\langle F, \sfP\rangle$ is the expectation of $F$ under $\sfP$, therefore, we also write $\langle F, \sfP \rangle = \E_\sfP[F(\eta)]$.

\medskip
We define the total variation norm on $\calP(\Omega)$ as follows
\begin{definition}
\label{def: total variation}
For any $\sfP, \mathsf{Q} \in \calP(\Omega)$, the total variation norm is defined by
\begin{align*}
\norm{\sfP - \mathsf{Q}}_{\rm{TV}} \coloneq \sup \Bigl\{ |\langle F, \sfP\rangle - \langle F, \mathsf{Q}\rangle| : F \in B_b(\Omega),\; \|F\|_\infty \le 1 \Bigr\}.
\end{align*}
% where $B_b(\Omega)$ is the set of bounded Borel functions.
\end{definition}

Since the forward Kolmogorov equation \eqref{eq: forward kolmogorov} holds for all $F \in B_b(\Omega)$, we can investigate what happens for specific observables $F$, e.g., $F \coloneq \err(q, \eta)$ for some fixed $q\in P$. In the next section, we use this strategy to obtain convergence results for local functions.

\section{Convergence: Local Functions}
\label{sec: convergence}
This section shows that the error function $\err{:} P \times \Omega \to \R^+$ converges to zero almost everywhere in the large-time limit. We make the following assumption on the local error function:
\begin{assumption}
\label{assumption: J}
The local error function $\err{:} P \times \Omega \to [0,+\infty)$ satisfies the following:
    \begin{enumerate}
    \item ({\it Boundedness}) There exists a $c_0>0$ such that $\err(p, \eta) \le c_0$ for all $(\eta,p) \in \Omega\times P$.
    \item ({\it Monotonicity}) For every $(\eta,p,y) \in \Omega\times P\times P$, it holds that
    \begin{align*}
        \err(p, \eta {\oplus} y) \le \err(p, \eta).
    \end{align*}
    \item ({\it Consistency}) For every $\eta=(\eta_1,\ldots,\eta_n,\phi,\ldots)\in \Omega_n$,
    \[
        \err(\eta_i, \eta) = 0,\qquad i=1,\ldots,n.
    \]
    \item ({\it Regularity}) For any $\eta\in\Omega$, the map $p\mapsto \err(p, \eta)$ is Lipschitz continuous with Lipschitz constant $L_\err$, independent of $\eta$.
    % \item $J(q,\cdot)$ is invariant under permutations, i.e., $\err(p, \eta) = \jmath(q,\Lambda(\eta))$ for some function $\jmath{:}P\times\varGamma(P)\to [0,+\infty)$.
\end{enumerate}
\end{assumption}

\begin{remark}
    In many practical cases, the local error function $\err(p, \cdot)$ is invariant under permutations of the points, i.e., $\err(p, \eta) = \jmath(p, \Lambda(\eta))$ for some function $\jmath{:} P \times \varGamma(P) \to [0, +\infty)$.
\end{remark}

In Theorem~\ref{theorem: convergence J}, we formulate the main statement of this section. Lemma~\ref{lemma: bounded integral} and Lemma~\ref{lemma: uniform continuity} are stepping stones to prove Theorem~\ref{theorem: convergence J}.

\medskip
The main statement of this section is:
\begin{theorem}\label{theorem: convergence J}
    Let $\err{:} P \times \Omega \to [0,+\infty)$ be a local error function satisfying Assumption~\ref{assumption: J}. Then,
    \[
        \lim_{t\to\infty} \int_\Omega |LJ(p,\eta)|\,\sfP_t(d\eta) = 0\qquad\text{for every $p\in P$.}
    \]
    In particular, if $\eta\mapsto LJ(p,\eta)$ is lower semicontinuous for every $p\in P$, then every accumulation point $\sfP_\infty$ of $(\sfP_t)_{t\ge 0}\subset \calP(\Omega)$ in the narrow topology satisfies
    \[
       \text{$\err(y, \eta) = 0$\;\; for\;\; $\lambda(\eta, dy) \Pinf(d\eta)$-almost every $(y, \eta) \in P \times \Omega$.}
    \]
\end{theorem}

% \begin{theorem}
% \label{theorem: convergence J}
%     Let $\err{:} P \times \Omega \to [0,+\infty)$ be any local error function satisfying Assumption~\ref{assumption: J}. Then, there exist $\sfP_\infty\in\calP(\Omega)$ such that $\sfP_t \rightharpoonup \sfP_\infty$
%     \[
%        \text{$\err(y, \eta) = 0$\;\; for\;\; $\lambda(\eta, dy) \Pinf(d\eta)$-almost every $(y, \eta) \in P \times \Omega$.}
%     \]
% \end{theorem}
%The statement means that $\err(y, \eta) = 0$ everywhere in $P \times \Omega$ except for subsets of $P \times \Omega$ with zero measure under $\lambda(\eta, dy) \otimes \Pinf$.
%The statement means that $\err(y, \eta)$ can only be $0$ in a subset of $P \times \Omega$ with measure $0$ under $\lambda(\eta,dy) \otimes \Pinf$. 
\begin{remark} In particular, if $\lambda(\eta, \cdot)$ is equivalent to the Lebesgue measure $\calL$, i.e. $\lambda(\eta, \cdot) \ll \calL$ and $\lambda(\eta, \cdot) \gg \calL$ for every $\eta \in \Omega$, then the result of Theorem~\ref{theorem: convergence J} implies $\err(y, \eta) = 0$ for $\calL \otimes \Pinf$-almost every $(y, \eta) \in P \times \Omega$.
\end{remark}

The following lemmas provide stepping stones to proving the first statement in Theorem~\ref{theorem: convergence J}. 

\begin{lemma}
\label{lemma: bounded integral}
Let $(\sfP_t)_{t \ge 0}$ be a solution to the Forward Kolmogorov equation. Then,
\begin{align*}
    \int_0^{\infty} \langle (LJ(p,\cdot))^-, \sfP_r\rangle \,dr \le c_0\qquad\text{for every $p\in P$}.
\end{align*}
Here, $(LF)^-$ denotes the negative part of $(LF)$, i.e., $(LF)^-(\eta) \coloneq |\min\{0, LF(\eta)\}|$. 
\end{lemma}

\begin{proof}
Since $(\sfP_t)_{t \ge 0}$ solves the Forward Kolmogorov equation, we have for $F_p(\eta) \coloneq \err(p, \eta)$ that
\begin{align*}
\E[F_p(\eta_t)] - \E[F_p(\eta_0)] &= \int_0^t\!\! \int_{\Omega} LF_p(\eta)\, \sfP_s(d\eta)\, ds,\\
&=\int_0^t\!\! \int_{\Omega} \Bigl[(LF_p(\eta))^{+} - (LF_p(\eta))^{-}\Bigr] \sfP_s(d\eta)\, ds,\\
&= - \int_0^t\!\! \int_{\Omega} (LF_p(\eta))^{-} \sfP_s(d\eta)\, ds.
\end{align*}
This last step follows from Assumption~\ref{assumption: J}(4). Hence, we conclude that
\begin{align*}
0 \le \int_0^t\!\! \int_{\Omega} (LF_p(\eta))^{-} \sfP_s(d\eta)\,ds \le \E[F_p(\eta_0)] \le c_0.
\end{align*}
The statement follows after sending $t$ to infinity.
\end{proof}

\begin{lemma}
\label{lemma: uniform continuity}
    The map $t \mapsto \int_{\Omega} (L\err(p, \eta))^{-} \sfP_t(d\eta)$ is uniformly continuous for any $p\in P$.
\end{lemma}
\begin{proof}
    As before, we set $F_p(\eta) \coloneq \err(p, \eta)$, $\eta\in\Omega$. Then
    \[
        -2c_0 \le L((L\err(p, \eta))^{-}) \le 2c_0\qquad\text{for every $\eta\in\Omega$}.
    \]
    Hence, 
    \[
        \left|\int_s^t \langle L((L\err(p, \eta))^{-}),\sfP_r\rangle\,dr\right| \le 2c_0|t-s|,
    \]
    therewith implying the differentiability of the map $t \mapsto \int_{\Omega} (L\err(p, \eta))^{-} \sfP_t(d\eta)$ with
    \[
        \left|\frac{d}{dt}\int_{\Omega} (L\err(p, \eta))^{-} \sfP_t(d\eta)\right| \le 2c_0,
    \]
    allowing us to conclude that it is uniformly continuous.
    % We note that the mapping is continuous since it is differentiable and has derivative:
    % \begin{align*}
    %     \frac{d}{dt} \int_\Omega (LF(\eta))^{-} \sfP_t(d\eta) &= \int_{\Omega} L((LF(\eta))^{-}) \sfP_t(d\eta).     
    % \end{align*}
    % Furthermore, we note that
    % \begin{align*}
    %     \left|\int_{\Omega} L((LF(\eta))^{-}) \sfP_t(d\eta)\right| &\le \int_{\Omega} \norm{L(LF)} \sfP_t(d\eta),\\
    %     &\le C_L \int_{\Omega} \norm{LF} \sfP_t(d\eta),\\
    %     &\le C_L^2 \norm{F}.
    % \end{align*}
    % Since the mapping $t \mapsto \int_{\Omega} (LF_p(\eta))^{-} \sfP_t(d\eta)$ has a bounded derivative, we can conclude it is uniformly continuous.
\end{proof}

A consequence of Lemmas~\ref{lemma: bounded integral} and \ref{lemma: uniform continuity} one may then conclude that $\lim_{t \to \infty} \langle (LJ(q,\cdot))^-, \sfP_t \rangle = 0$. On the other hand, the compactness of $\Omega$ implies that any family of probability measures in $\calP(\Omega)$ is tight, thus asserting the existence of accumulation points for the sequence $(\sfP_t)_{t\ge 0}\subset \calP(\Omega)$.

% From Theorem~\ref{theorem: convergence}, we know that a limiting object $\Pinf$ exists. To show that we can interchange the limit and the integral, we use the following theorem: 

Now we are in a position to prove Theorem~\ref{theorem: convergence J}.

\begin{proof}[Proof of Theorem~\ref{theorem: convergence J}]
As mentioned, Lemmas~\ref{lemma: bounded integral} and \ref{lemma: uniform continuity} allows us to conclude that \cite{kelman1960conditions}
\[
    \lim_{t\to\infty} \int_\Omega |LJ(q,\eta)| \,\sfP_t(d\eta) = 0\qquad\text{for every $p\in P$},
\]
where we used the fact that $(LJ(q,\eta))^+=0$ for every $\eta\in\Omega$.

As for the second part, we consider any accummulation point $\sfP_\infty\in\calP(\Omega)$ and a subsequence $(\sfP_{t_n})_{n\ge 1}$ with $t_n\to \infty$ as $n\to\infty$ such that $\sfP_{t_n}\rightharpoonup \sfP_\infty$. Since $\eta\mapsto LJ(q,\eta)$ is assumed to be lower semicontinuous, we conclude that
\[
    \int_\Omega (LJ(q,\eta))^-\,\sfP_\infty(d\eta) \le \liminf_{n\to\infty} \int_\Omega (LJ(q,\eta))^- \,\sfP_{t_n}(d\eta) = 0.
\]
% As a result of , we can now conclude that for any convergent subsequence (not relabeled) $(\sfP_t)_{t \ge 0}$, we have that
% \begin{align*}
%     \int_\Omega \int_P \bigl[\err(p, \eta) - J(p, \eta {\oplus} y)\bigr] \lambda(\eta, dy) \Pinf(d\eta) = \lim_{t \to \infty} \int_\Omega (L\err(p, \eta))^- \sfP_t(d\eta) = 0.
% \end{align*}
By Assumption~\ref{assumption: J}(4), we then deduce that
\begin{align*}
    \err(p, \eta) = J(p, \eta {\oplus} y) \qquad\text{for $\lambda(\eta, dy)\Pinf(d\eta)$-almost every $(y, \eta) \in P \times \Omega$.}
\end{align*}
In particular, for $p = y$, Assumption~\ref{assumption: J}(2) gives
\begin{align*}
    0 = J(y, \eta {\oplus} y) = J(y, \eta)\qquad\text{for $\lambda(\eta, dy)\Pinf(d\eta)$-almost every $(y, \eta) \in P \times \Omega$},
\end{align*}
therewith concluding the proof.
\end{proof}

% In this section, we showed that $\err(p, \eta_t) \to 0$ for $\lambda(\eta,dp) \Pinf(d\eta)$-almost every $(p, \eta) \in P \times \Omega$. 
In this following section, we derive convergence results for the class of global error functions, $\scrG{:} \Omega \to \R^+$, and derive convergence rates.

%\section{Convergence rates}
%\label{sec: convergence rates}
%\input{Rates}

\section{Convergence: Global Functions}
\label{sec: average improvement}
This section considers a class of global functions $\scrG{:} \Omega \to [0,+\infty)$. We replace Assumption~\ref{assumption: J} with the following assumption
\begin{assumption}
    \label{assumption: average error}
    Let $\scrG{:} \Omega \to [0,+\infty)$, then $\scrG$ satisfies:
    \begin{enumerate}
        \item ({\it Boundedness}) \label{item: boundedness} There exists a $c_0 \in [0,+\infty)$ such that $\scrG(\eta) \le c_0$ for all $\eta \in \Omega$.
        \item ({\it Monotonicity}) \label{item: monotonicity} For every $y \in P$, and $\eta \in \Omega$, it holds that 
        \begin{align*}
            \scrG(\eta {\oplus} y)\le \scrG(\eta).
        \end{align*}
        \item \label{item: saturation} ({\it Saturation property}) For any $\eta \in \Omega$,
        \[
         \text{$\scrG(\eta) = \scrG(\eta {\oplus} y )$\, $\lambda(\eta, dy)$-almost every $y \in P$\;\; implies\;\;} \scrG(\eta) = 0.
        \]
        \myitem[(3')]\label{item: improvement factor}  ({\it Improvement factor}) There exists a $\gamma \in (0,1), \delta > 0$, $\beta \in [0,1]$, and for every $\eta\in\Omega$ there exists a set $B(\eta) \subset P$ with $\lambda(\eta, B(\eta)) \ge \delta$, such that
        \[
            \int_{B(\eta)} \bigl(\scrG(\eta) - \scrG(\eta {\oplus} y)\bigr) \, \lambda(\eta, dy) \ge \frac{\gamma \delta}{\sfN^\beta(\eta)} \scrG(\eta).
        \]
        
    \end{enumerate}
    We either consider item~\eqref{item: saturation} or \ref{item: improvement factor}. We note that \ref{item: improvement factor} implies \eqref{item: saturation}.
    
    The class of global error functions includes the average function of local error functions, i.e.,
    \[
        \scrG(\eta) = \int_P \err(q,\eta)\,dq\qquad\text{for some local function $\err$.}
    \]
    It also includes the common loss function 
    \[
        \mfL(\eta) = \sup_{q \in P} \err(q, \eta),
    \]
    and the previous case: $\scrG(\eta) = \err(q,\eta)$ for some arbitrary but fixed $q\in P$.
\end{assumption}

Theorem~\ref{theorem: convergence J} may be adapted to obtain a similar result for global error functions, as shown in the following lemma. 
% We derive a new convergence result before presenting convergence rates.
\begin{theorem}
\label{theorem: average convergence}
Let $\scrG{:} \Omega \to [0,+\infty)$ be a global error function satisfying Assumption~\ref{assumption: average error} \eqref{item: boundedness}--\eqref{item: saturation}. Then,
    \[
        \lim_{t\to\infty} \int_\Omega |L\scrG(\eta)|\,\sfP_t(d\eta) = 0.
    \]
    In particular, if $\scrG$ is lower semicontinuous, then every accumulation point $\sfP_\infty$ of $(\sfP_t)_{t\ge 0}\subset \calP(\Omega)$ in the narrow topology satisfies
\begin{align*}
    \scrG(\eta) = 0 \qquad \text{for} \; \Pinf\text{-almost every} \, \eta \in \Omega.
\end{align*}
\end{theorem}
\begin{proof}
    The proof of this lemma is analogous to the proof of Theorem~\ref{theorem: convergence J} with $F(\eta) = \scrG(\eta)$.
    
    Analogously to the proof of Theorem~\ref{theorem: convergence J}, we can conclude that
    \[
        \lim_{t \to \infty} \int_\Omega (L\scrG(\eta))^- \sfP_t(d\eta) = 0.
    \]
    For an accumulation point $\sfP_\infty\in\calP(\Omega)$ and a subsequence $(\sfP_{t_n})_{n\ge 1}$ with $t_n\to \infty$ as $n\to\infty$ such that $\sfP_{t_n}\rightharpoonup \sfP_\infty$, we find
    \begin{align*}
        \int_\Omega (L\scrG(\eta))^-\Pinf(d\eta) \le \liminf_{n \to \infty} \int_\Omega (L\scrG(\eta))^- \sfP_{t_n}(d\eta) = 0.
    \end{align*}
    Therefore,
    \begin{align*}
        \int_\Omega \int_P \big[\scrG(\eta)-\scrG(\eta {\oplus} y)\big] \lambda(\eta, dy) \, \Pinf(d\eta) = 0,
    \end{align*}
    i.e., $\scrG(\eta) = \scrG(\eta {\oplus} y)$ for $\lambda {\otimes} \sfP_{\!\!\infty}$-almost every $(y, \eta) \in P \times  \Omega$. By the saturation proper of $\scrG$ (cf.\ Assumption~\ref{assumption: average error}\,\eqref{item: saturation}), we can then conclude that the assertion holds.
\end{proof}

Next, we formulate a stronger result than Theorems~\ref{theorem: convergence J} and \ref{theorem: average convergence} under the improvement factor condition on $\scrG$ (cf.\ Assumption~\ref{assumption: average error}\ref{item: improvement factor} in the sense that (1) we obtain explicit convergence rates, and (2) lower semicontinuity of $\scrG$ is no longer required to assert that $\scrG(\eta_t)\approx 0$ for times $t\gg 1$.

\begin{theorem}
\label{thm: convergence rate average}
    Let $(\eta_t)_{t\ge 0}$ be the process generated by~\eqref{eq: generator} with transition kernel $\lambda{:} \Omega \times P \to [0,+\infty)$ satisfying Assumption~\ref{assumption: rate function}. Further, let $\scrG{:} \Omega \to [0,+\infty)$ be a global error function satisfying Assumption~\ref{assumption: average error} for some $\gamma \in (0,1), \delta > 0, \beta \in [0,1]$. Then for every $\eps > 0$, there exists a constant $c_\scrG > 0$, independent of $\eps$, such that
    \[
        \p\big(\scrG(\eta_t) > \eps\big) \le \frac{c_\scrG}{\eps}\E[\scrG(\eta_0)]\, \theta_\beta(t)\qquad \text{for} \; t\ge 1,
    \]
    with
        \begin{align*}
        \theta_\beta(t) =\begin{cases}
        e^{-\gamma \delta t} & \text{for $\beta = 0$,}\\
        t^{1-\frac{1}{\beta}} & \text{for $\beta\in (0,1)$,} \\
       \frac{1}{\log(1 + t)} & \text{for $\beta=1$.}
       \end{cases}
    \end{align*}
    In particular, $\scrG(\eta_t)$ converges in probability to $0$ as $t\to\infty$.
\end{theorem}

\begin{proof}
    By Markov's inequality, we have that
    \begin{align*}
        \p\bigl(\scrG(\eta_t) > \eps\bigr) \le \frac{1}{\eps}\E[\scrG(\eta_t)].
    \end{align*}
    Hence, we look for an upper bound of $E_t \coloneq \E[\scrG(\eta_t)]$.
    For $\beta = 0$, we have that
    \begin{align*}
        -L\scrG(\eta) =  \int_P \big[\scrG(\eta) - \scrG(\eta {\oplus} y)\big]\lambda(\eta, dy) &\ge \gamma \int_{B(\eta)} \scrG(\eta) \, \lambda(\eta, dy) \ge \gamma \delta \scrG(\eta).
    \end{align*}
    Hence,
    \[
        \frac{d}{dt}E_t \le - \gamma \delta E_t.  
    \]
    So by Gronwall's inequality, we conclude
    \begin{align*}
        E_t \le E_0 e^{-\gamma \delta t}.
    \end{align*}
    % Hence, we conclude for $\beta = 0$, that
    % \[\p(\scrG(\eta_t) < \eps) \ge 1- \frac{1}{\eps} \E[G(\eta_0)]  e^{-\gamma \delta t}.\]
    
    For $\beta \in (0, 1]$, we have
        \begin{align*}
        \int_P \big[\scrG(\eta) - \scrG(\eta {\oplus} y)\big] \lambda(\eta, dy) &\ge \int_{B(\eta)} [\scrG(\eta) - \scrG(\eta {\oplus} y)] \lambda(\eta, dy)
        \ge \frac{\gamma\delta}{\sfN^\beta(\eta)} \scrG(\eta).
    \end{align*}
    Let $\mu \coloneq \gamma \delta$, then we have for an arbitrary $K\in[1,+\infty)$,
    \begin{align*}
        \frac{d}{dt} E_t &\le -\mu \E\bigg[\frac{\scrG(\eta_t)}{\sfN^\beta(\eta_t)}\bigg] = - \mu \E\bigg[\frac{\scrG(\eta_t)}{\sfN^\beta(\eta_t)}\mathbbm{1}_{\{\sfN(\eta_t) > K\}}\bigg] - \mu\E\bigg[\frac{\scrG(\eta_t)}{\sfN^\beta(\eta_t)}\mathbbm{1}_{\{\sfN(\eta_t)\le K\}}\bigg],\\
        &\le - \frac{\mu}{K^\beta} \E\left[\scrG(\eta_t) \mathbbm{1}_{\{\sfN(\eta_t) \le K\}} \right] = - \frac{\mu}{K^\beta} \E\left[\scrG(\eta_t)\right] + \frac{\mu}{K^\beta} \E\left[\scrG(\eta_t) \mathbbm{1}_{\{\sfN(\eta_t) > K\}} \right],\\
        &\le - \frac{\mu}{K^\beta}  E_t + \frac{c_0 \mu  (1+t)}{K^{1+\beta}} \eqcolon g_t(K,E_t).
    \end{align*}
    In the last step, we used Assumption~\ref{assumption: average error}\eqref{item: boundedness} and the fact that
    \[
        \E\left[\mathbbm{1}_{\{\sfN(\eta_t) > K\}} \right] = \p\bigl(\sfN(\eta_t) > K\bigr) \le \frac{1+t}{K}.
    \]
    We now determine for which $K\mapsto g_t(K,E_t)$ is minimized. A stationary point is given by
    \[
        K_\circ = \frac{1+\beta}{\beta} \frac{c_0}{E_t}(1+t) \in [1,\+\infty).
    \]
    Since, $\partial_K^2 g_t(K_\circ,E_t) = \beta\mu E_t/K_\circ^{2+\beta} > 0$, the stationary point $K_\circ$ is a minimizer.
    % Differentiating the right-hand side in $K$,  we obtain
    % \begin{align*}
    %     \beta \frac{\mu}{K^{\beta + 1}} E -(\beta+1) \frac{c_0 \mu (1+t)}{K^{\beta+2}}, 
    % \end{align*}
    % \begin{align*}
    %     -\beta(1+\beta) \frac{\mu}{K^{2+\beta}} E + (1+\beta)(2+\beta) \frac{c_0 \mu (1+t)}{K^{3+\beta}} &= \frac{\beta\mu}{K^{2+\beta}}\left[-(1+\beta)E + (2+\beta)\frac{1+\beta}{\beta}\frac{c_0(1+t)}{K}\right] \\
    %     &= \frac{\beta\mu}{K^{2+\beta}}E
    % \end{align*}
    % which equals $0$ for $K = \frac{(\beta+1)c_0 t}{\beta E}$.
    % Let $\alpha_\beta \coloneq \left(\mu\, c_0^{-\beta} \Big(\frac{\beta}{\beta+1}\Big)^{\beta} \Big(1-\frac{\beta}{\beta+1}\Big)\right)$.
    Then, we have
    \[
        \frac{d}{dt}E_t \le g_t(K_\circ,E_t) = - \frac{\alpha_\beta}{(1+t)^{\beta}}E_t^{1+\beta},\qquad \alpha_\beta \coloneq \frac{\mu c_0^{-\beta}}{1+\beta}\left(\frac{\beta}{1+\beta}\right)^\beta.
    \]
    Solving the differential inequality for $E_t$ yields
    \[
        E_t \le E_0\left( 1+ \alpha_\beta E_0^\beta \frac{\beta}{1-\beta}\Bigl((1+t)^{1-\beta}-1\Bigr)\right)^{-1/\beta} \le E_0\,\mathcal{O}\bigl(t^{1-\frac{1}{\beta}}\bigr)|_{t\to\infty}.
    \]
    % \begin{align*}
    %     -\frac{1}{\beta} \frac{d}{dt} E^{-\beta} = E^{-(\beta+1)}\frac{d}{dt} E \le - \alpha_\beta \frac{1}{t^\beta}.
    % \end{align*}
    % Therefore, for $\beta \in (0,1)$, it holds that
    % \begin{align*}
    %     &\frac{d}{dt} E^{-\beta} \ge \beta \alpha_\beta t^{-\beta}, \\
    %     \Longleftrightarrow \, & E^{-\beta} - E_0^{-\beta} \ge \frac{\beta \alpha_\beta}{1-\beta} t^{1-\beta},\\
    %     \Longrightarrow \, &E^{-\beta}  \ge E_0^{-\beta} + \frac{\beta}{1-\beta} \alpha_\beta t^{1-\beta},\\
    %     \Longrightarrow \, &E^\beta \le \frac{1}{E_0^{-\beta} + \frac{\beta}{1-\beta} \alpha_\beta t^{1-\beta}} \le \frac{1}{ \frac{\beta}{1-\beta} \alpha_\beta t^{1-\beta}}\\
    %     \Longrightarrow \, &\E[\scrG(\eta_t)] \le \left(\frac{1-\beta}{\beta \alpha_\beta} \right)^{\frac{1}{\beta}}t^{-\frac{1}{\beta} + 1}.
    % \end{align*}
    % Hence, we conclude for $\beta \in (0,1)$ that 
    % \begin{align*}
    %     \p(\scrG(\eta_t) < \eps) \ge 1 - \frac{1}{\eps} \left(\frac{1-\beta}{\beta \alpha_\beta}\right)^{\frac{1}{\beta}} t^{-\frac{1}{\beta} + 1}.
    % \end{align*}
    
    Finally, for $\beta = 1$, we deduce
    \[
        E_t \le E_0 \Bigl(1+\alpha_1E_0 \log(1+t)\Bigr)^{-1} = E_0\,\mathcal{O}\bigl(\bigl(\log(1+t)\bigr)^{-1}\bigr)|_{t\to\infty},
    \]
    % \begin{align*}
    %     &\frac{d}{dt} E^{-1} \ge \alpha_1 t^{-1} \ge \alpha_1 (1+t)^{-1},\\
    %     \Longleftrightarrow \, & E^{-1} - E_0^{-1} \ge \alpha_1 \log(1+t),\\
    %     \Longrightarrow\, & E \le \frac{1}{E_0^{-1} + \alpha_1 \log(1+t)} \le \frac{1}{\alpha_1 \log(1+t)}.         
    % \end{align*}
    % Hence, we conclude for $\beta = 1$ that 
    % \begin{align*}
    %     \p(\scrG(\eta_t) < \eps) \ge 1 - \frac{1}{\alpha_1 \eps} \frac{1}{\log(1+t)}.
    % \end{align*}
    Thereby concluding the proof.
\end{proof}

\section{Pedagogical example: Interpolation in 1D}
\label{sec: interpolation}
One example of a COP is piecewise linear interpolation. The main result in this Section (Theorem~\ref{theorem: interpolation C2}) states that we have exponential convergence of the $L_1$ error of piecewise linear interpolation of $C^2$-functions under certain assumptions of transition kernel $\lambda$.

%\subsection{Piecewise linear interpolation}
%In the first two examples, we approximate a twice continuously differentiable strongly convex function $f{:} [a,b] \to \R$ with a piecewise linear function $\mathfrak{I}[f]$. In the third example, we consider a function $f \in C^2([a,b])$.
Before stating the main result, we define the piecewise linear interpolation. The piecewise linear interpolation function depends on a set of nodes $x_0 < x_1 < \ldots < x_{n+1}$, with $x_0 = a$ and $x_{n+1} = b$. We use stochastic greedy methods to find the interpolation nodes $\{x_1, \ldots, x_n\}$ in the parameter set $P=[a,b]$. Note that the points $x_0 = a$ and $x_{n+1} = b$ are not part of the nodes selected by the algorithm. Typically, stochastic greedy methods do not lead to an ordered list $\eta = (x_1, \ldots, x_n, \phi, \ldots)$. Therefore, we define the following map to order the elements of $\eta\in\Omega$.
\begin{definition}
    \label{def: order mapping}
    The \emph{order mapping} $\mfS{:} \Omega \to \Omega$ is defined as
    \begin{align*}
        \mfS(\eta) \coloneq (\eta_{\sigma(1)}, \ldots, \eta_{\sigma(\sfN(\eta))}, \phi,\ldots),
    \end{align*}
    such that $x_i = \eta_{\sigma(i)}$ for $i \in \{1, \ldots, \sfN(\eta)\}$ satisfies $x_i \le x_{i+1}$ for all $i \in \{0, \ldots, \sfN(\eta)\}$.  
\end{definition}

An important property of $\mfS$ is given in the following lemma, whose proof is provided in Appendix~\ref{app: compactness} for completeness.

\begin{lemma}\label{lemma: order map}
    The order mapping $\mfS{:} \Omega \to \Omega$ is continuous.
\end{lemma}

With an ordering of $\eta$, we may now define the linear approximation of a function based on $\eta$.

\begin{definition}[Linear interpolation]
    Let $f{:}[a,b]\to \R$ be a function. For any $\eta\in\Omega_n$, the piecewise linear approximation of $f$ relative to $\eta$ is defined as
\begin{align}
\label{eq: linear approximation L}
    \mathfrak{I}_\eta[f](x) = f(x_k) + \frac{f(x_{k+1}) - f(x_k)}{x_{k+1} - x_k} (x - x_k) \qquad \text{for} \, x \in [x_k, x_{k+1}],
\end{align}
where $\mfS(\eta) = (x_1,\ldots,x_n,\phi,\ldots)$.
\end{definition}
% For a configuration $\eta = (x_1, \ldots, x_n, \phi, \ldots)$, with $x_1 < \ldots < x_n$, the linear approximation is defined as

% We use stochastic greedy methods to find the interpolation nodes $\{x_1, \ldots, x_n\}$. 

\begin{remark}
    We note that we can rewrite $\mathfrak{I}_{\eta}[f]$ as
    \begin{align*}
        \mathfrak{I}_{\eta}[f](x) = f(a) \sigma_a(x) + \sum_{i=1}^{n} f(x_i) \sigma_{x_i}(x) + f(b)\sigma_b(x),\qquad \mfS(\eta) = (x_1,\ldots,x_n,\phi,\ldots),
    \end{align*}
    with
    \begin{align*}
        \sigma_{x_i}(x) = 
        \begin{cases}
            \frac{x-x_{i-1}}{x_i - x_{i-1}}\quad &\text{for}\, x \in [x_{i-1}, x_i],\\
            \frac{x_{i+1}-x}{x_{i+1} - x_i}\quad &\text{for} \, x \in [x_i, x_{i+1}],\\
            0\quad &\text{otherwise}.
        \end{cases}
    \end{align*}
    We note that this expression is summable as $\sfN(\eta) \to \infty$, since
    \begin{align*}
        \sum_{i=1}^{n} f(x_i) \sigma_{x_i}(x) \le \norm{f}_{\rm{sup}} \sum_{i=1}^{n} \sigma_{x_i}(x) \le \norm{f}_{\rm{sup}}.
    \end{align*}
\end{remark}

\begin{remark}
    In the definition of the linear approximation~\eqref{eq: linear approximation L}, we assume that the interpolation nodes are strictly increasing. We note that Definition~\ref{def: order mapping} does not require the same for the points $\{x_{\sigma(1)}, \ldots,x_{N(\eta)}\}$. However, for many stochastic greedy methods, the probability of sampling the same points twice is $0$.
\end{remark}

We consider the following local error function
\begin{align}
    \label{eq: pointwise error interpolation}
    \err(x, \eta) = |\mathfrak{I}_{\eta}[f](x) - f(x)|,\qquad x\in P\coloneq [a,b].
\end{align}
We are interested in the convergence of the following global error function.
\begin{align}
    \label{eq: error interpolation}
    \scrG(\eta) = \int_P J(x,\eta)\, dx = \|\mathfrak{I}_{\eta}[f]-f\|_{L^1}. 
\end{align}

For this global error function, the following theorem holds
\begin{theorem}
    The global error function $\scrG{:}\Omega\to [0,+\infty)$ defined in \ref{eq: error interpolation} is continuous.
\end{theorem}
\begin{proof}
    We note that it is sufficient to prove that $J(x,\eta)$ is continuous.
    Let $\eta \in \Omega$, let $(\eta^k)_{k\in\N}$ be a sequence in $\Omega$ such that $\mfd(\eta^k,\eta)\to 0$ as $k \to \infty$. Then there exists a $K \in \N$, such that $\sfN(\eta) = \sfN(\eta^k)$ for all $k \ge K$.
    Let $\mathfrak{S}(\eta) = (x_1, \ldots, x_n, \ldots)$ and $\mathfrak{S}(\eta^k) = (x_1^k, \ldots, x_n^k, \ldots)$ for all $k \ge K$, and define a function $T^k{:} [a,b] \to [a,b]$ with the property that $T^k([x_i^k, x_{i+1}^k]) = [x_i, x_{i+1}]$ for all $i \in \{0, \ldots, n\}$ and $k \ge K$.
    This function is given by
    \[
        T^k(x) = \frac{x_{i+1}^k -x}{x_{i+1}^k - x_i^k} x_i + \frac{x - x_i^k}{x_{i+1}^k - x_i^k} x_{i+1}\quad \text{for} \, x \in [x_i^k, x_{i+1}^k].
    \]
    It holds that
    \begin{align*}
        \mathfrak{I}_{\eta^k}[f](x) - \mathfrak{I}_\eta [f](x) = \underset{(A)}{\underbrace{\mathfrak{I}_{\eta^k}[f](x) - \mathfrak{I}_{\eta^k}[f]\big((T^k)^{-1}(x)\big)}} + \underset{(B)}{\underbrace{\mathfrak{I}_{\eta^k}[f]\big((T^k)^{-1}(x)\big) - \mathfrak{I}_\eta[f](x)}}. 
    \end{align*}
    For $(A)$, we have that
    \begin{align*}
        (A) \le \|\mathfrak{I}_{\eta^k}[f]\|_{\mathrm{sup}} |x-(T^k)^{-1}(x)| \le \|f\|_{\mathrm{sup}}|x-(T^k)^{-1}(x)|.
    \end{align*}
    Observing that
    \[
        (T^k)^{-1}(x) = \frac{x_{i+1} - x}{x_{i+1}- x_i} x_i^k + \frac{x-x_i}{x_{i+1}-x_i} x_{i+1}^k \quad \text{for} \, x \in [x_i, x_{i+1}],
    \]
    we then obtain the estimate
    \begin{align*}
        |x - (T^k)^{-1}(x)| &= \biggl|\frac{(x_{i+1} -x_i) x - (x_{i+1}-x) x_i^k - (x-x_i) x_{i+1}^k}{x_{i+1}-x_i}\biggr|\\
        &= \biggl|\frac{(x-x_i)(x_{i+1} - x_{i+1}^k) + (x_{i+1} - x)(x_i - x_i^k)}{x_{i+1}-x_i} \biggr|\\
        &\le \max\{|x_{i+1} -x_{i+1}^k|, |x_i - x_i^k|\}.
    \end{align*}
    Hence, we conclude
    \begin{align*}
        (A) \le \|f\|_{\mathrm{sup}} \max\{|x_{i+1} -x_{i+1}^k|, |x_i - x_i^k|\}.
    \end{align*}
    Moreover, by definition, we have 
    \[\mathfrak{I}_{\eta^k}[f](x) \stackrel{(*)}{=} \mathfrak{I}_\eta[f \circ (T^k)^{-1}](T^k(x)).\]
    Hence, we have
    \begin{align*}
        (B) &\stackrel{(*)}{=} |\mathfrak{I}_\eta[f \circ (T^k)^{-1}](x) - \mathfrak{I}_\eta[f](x)|\\
        &= |\mathfrak{I}_\eta[f\circ (T^k)^{-1} - f](x)|\\
        &= \biggl|\frac{x - x_i}{x_{i+1} - x_i} (f\circ (T^k)^{-1} - f)(x_{i+1}) + \frac{x_{i+1} -x}{x_{i+1}-x_i} (f\circ (T^k)^{-1} - f)(x_i)\biggr|,\\
        &= \biggl|\frac{x - x_i}{x_{i+1} - x_i} (f(x_{i+1}^k)-f(x_{i+1}) + \frac{x_{i+1} -x}{x_{i+1}-x_i} (f(x_i^k)-f(x_i))\biggr|,\\
        & \le \|f\|_{\mathrm{sup}} \max\{|x_{i+1} -x_{i+1}^k|, |x_i - x_i^k|\}.
    \end{align*}
    Together, we obtain
    \[
        \|J(\cdot,\eta^k)-J(\cdot,\eta)\|_{\sup} = \|\mathfrak{I}_{\eta^k}[f] - \mathfrak{I}_\eta [f]\|_{\sup} \le 2\|f\|_{\sup} \max_{i=1,\ldots,n} |x_i^k-x_i|.
    \]
    We conclude that this latter expression goes to zero as $k \to \infty$ since $\mathfrak{S}$ is continuous. Therefore, both $J(x, \cdot)$ and $\scrG(\cdot)$ are continuous as well by the Dominated Convergence Theorem.
\end{proof}

To formulate our main result, we need to define the set $B_\mu(\eta)$.
\begin{definition}
\label{def: B(eta)}
Let $\mu \in (0, 1/2)$, $\eta \in \Omega_n$, $x_0 = a, x_{n+1} = b$, and $\mfS(\eta) = (x_1, \ldots, x_n, \phi, \ldots)$. For each $k\in\{0,\ldots,n\}$, we set $I_k^\mu \coloneq (x_k + \mu(x_{k+1}-x_k), x_{k+1} - \mu (x_{k+1}-x_k))$. Then,
\begin{align}
\label{eq: B(eta)}
B_\mu(\eta) \coloneq \bigl\{x \in [a,b]: x \in I^\mu_k,\; k \in \{0, \ldots, n\}\bigr\}.
\end{align}
\end{definition}
We make the following assumption about the transition kernel, which we later show to be true for uniform sampling and R-PDM in Section~\ref{sec: rpdm}.
\begin{property}
    \label{assumption: interpolation transition kernel}
    Let $\mu \in (0,1/2).$ The transition kernel $\lambda: \Omega \times \calB_{\R^d}\to [0,+\infty)$ satisfies 
    \[
    \lambda(\eta, B_\mu(\eta)) \ge 1 - 2\mu \quad \text{for all}\, \eta \in \Omega. 
    \]
\end{property}

We have the following main result on the convergence.
\begin{theorem}
\label{theorem: interpolation C2}
    Let $f \in C^2([a,b])$ with $c_f \coloneq \| f''\|_{\mathrm{sup}}$. Further, let $\mu \in (0, 1/2),$ and let $(\eta_t)$ be the process generated by the transition kernel $\lambda$ satisfying Property~\ref{assumption: interpolation transition kernel}.
    %by either $\lambda_{\mathrm{unif}}$ or $\lambda_{\mathrm{rpdm}}$, and $\mu \in (0, 1/2)$. 
    Then there exists a constant $m_\scrG>0$ such that for any $\eps > 0$ and $\alpha > c_f$,
    \begin{align*}
        \p\bigl(\scrG(\eta_t) > \eps\bigr) \le \frac{m_\scrG}{\eps} e^{-\gamma \delta t}\qquad \text{for} \; t\ge 1,
    \end{align*}
    with $\delta \coloneq 1-2\mu$, $\gamma \coloneq \mu \frac{\alpha - c_f}{\alpha + c_f}$.
\end{theorem}

To prove this theorem, we first show that Assumption~\ref{assumption: average error}\eqref{item: boundedness}--\ref{item: improvement factor} holds for strongly convex $C^2$-functions. With this, one then deduces that the convergence result also holds for $C^2$-functions. The idea behind this is as follows: Since $f \in C^2([a,b])$, setting $h_\alpha \coloneq \alpha |x|^2/2$, we have that the function
    \[
    x \mapsto f_\alpha(x)\coloneq f(x) + h_\alpha(x)\qquad\text{is strongly convex for $\alpha > c_f$}.
    \]
Therefore, $f = f_\alpha - h_\alpha$ where both $f_\alpha$ and $h_\alpha$ are strongly convex. 
In this way, we find
\begin{align*}
    \|\mathfrak{I}_\eta[f]-f\|_{L_1} &= \|(\mathfrak{I}_\eta[f_\alpha]-f_\alpha) - (\mathfrak{I}_\eta[h_\alpha] -  h_\alpha)\|_{L_1} \\
    &\le \|\mathfrak{I}_\eta[f_\alpha]-f_\alpha\|_{L_1} + \|\mathfrak{I}_\eta[h_\alpha] -  h_\alpha\|_{L_1}.
\end{align*}
We have the following result for strongly convex function $f \in C^2([a,b])$
\begin{lemma}
\label{lemma: assumptions satisfied}
    Let $f\in C^2([a,b])$ be a strongly $m$-convex function, i.e.,
    \[
        f((1-r)x + ry) \le (1-r)f(x) + rf(y) - \frac{m}{2}r(1-r)|x-y|^2\qquad\text{for every $r\in[0,1]$}.
    \]
    Let $\scrG{:} \Omega \to [0,+\infty)$ be given by \eqref{eq: error interpolation}. Let $\mu \in (0, 1/2),$ and let the transition kernel $\lambda$ satisfy Property~\ref{assumption: interpolation transition kernel}. Then Assumption~\ref{assumption: average error}\eqref{item: boundedness}--\ref{item: improvement factor} is satisfied.
    
    %Then the transition kernels $\lambda_{\mathrm{unif}}$ and $\lambda_{\rm{rpdm}}$ satisfy Assumption~\ref{assumption: average error}.
\end{lemma}
To prove this lemma, we have to check that Assumption~\ref{assumption: average error}\eqref{item: boundedness}--\ref{item: improvement factor} is satisfied. We prove this in steps. First, we formulate a lemma stating that the local error function 
\[
    \err{:} [a,b] \times \Omega \to [0,+\infty),\quad (p,\eta)\mapsto \err(p, \eta) \coloneq |\mathfrak{I}_\eta[f](p) - f(p)|,
\]
satisfies Assumption~\ref{assumption: J}.

\begin{lemma}
\label{lemma: interpolation error decreasing}
    Let $f\in C^2[a,b]$ be a strongly $m$-convex function with $
        0 < m \le f'' \le M$ on $[a,b]$. Then the following holds:
    \begin{enumerate}
        \item ({\it Consistency}) $\err(x_i, \eta) = 0$ for any $\eta \in \Omega_n$ with $\mfS(\eta) =(x_1,\cdots,x_n,\phi,\ldots)$.
        \item For any $\eta\in\Omega$ with $\mfS(\eta) =(x_1,\cdots,x_n,\phi,\ldots)$, we have for $p\in (x_k,x_{k+1})$,
        \[
            0\le \frac{m}{2}(x_{k+1}-p)(p-x_k) \le \mathfrak{I}_\eta[f](p) - f(p) \le \frac{M}{2}(x_{k+1} -p)(p-x_k).
        \]
        \item ({\it Monotonicity}) Let $y\in (x_k,x_{k+1})$. Then, for any $p\in[a,b]$,
        \[
            J(p,\eta) - J(p,\eta{\oplus}y)\ge \frac{m}{M}\biggl[\frac{x_{k+1} - y}{x_{k+1} - p} \mathbf{1}_{(x_k,y)}(p) + \frac{y-x_k}{p-x_k} \mathbf{1}_{(y,x_{k+1})}(p)\biggr]J(p,\eta) \ge 0.
        \]
    \end{enumerate}
    In particular, $J$ is a local error function satisfying Assumption~\ref{assumption: J}.
\end{lemma}
\begin{proof}
Let $\eta \in \Omega_n$ and $p \in [a, b]$. Further, let $\mfS(\eta) = (x_1, \ldots, x_n, \phi, \ldots)$, $x_0 = a$, and $x_{n+1} = b$. Then $p \in [x_k, x_{k+1}]$ for some $k \in \{0, \ldots, n\}$.

Firstly, note that if $p = x_k$, for any $k \in \{0, \ldots, n\}$, then $\err(p, \eta) = 0$, by the definition of the interpolation operator, which yields (1). 

\medskip
As for (2), we use the strong $m$-convexity of $f$ to deduce
\[
    f(p) \le \frac{x_{k+1}-p}{x_{k+1}-x_k}f(x_k) + \frac{p-x_k}{x_{k+1}-x_k}f(x_{k+1}) - \frac{m}{2}(x_{k+1}-p)(p-x_k),
\]
and from which we obtain
\begin{align*}
    \mathfrak{I}_\eta[f](p) - f(p) &= f(x_k) + \frac{f(x_{k+1}) - f(x_k)}{x_{k+1} - x_k} (p - x_k) - f(p) \\
    &\ge \frac{m}{2}(x_{k+1}-p)(p-x_k).
\end{align*}
As for the upper bound, we use Taylor's formula to obtain
\begin{align*}
    % \frac{m}{2}(x_{k+1}-p)(p-x_k) \le 
    \mathfrak{I}_\eta[f](p) - f(p) 
    % &= \frac{f(x_{k+1}) - f(x_k)}{x_{k+1} - x_k} (p-x_k) + f(x_k) - f(p),\\
    \le \frac{M}{2}(x_{k+1} -p)(p-x_k).
\end{align*}
Together, these yield the assertion for every $p\in(x_k,x_{k+1})$.

\medskip
We now prove (3): Suppose $y\in I_k\coloneq (x_k, x_{k+1})$ for some $k\in\{0,\ldots,n\}$.

\smallskip
\paragraph{\emph{Case 1:} $p\notin I_k$} In this case, we simply have $\err(p, \eta {\oplus} y)=\err(p, \eta)$ since the changes in the error only occurs in the interval $I_k$.

\smallskip
\paragraph{\emph{Case 2:} $p=y$} Due to (1), we have that $J(p,\eta{\oplus}y) = 0 \le J(p,\eta)$.

\smallskip
\paragraph{\emph{Case 3:} $p\in I_k,\; p\ne y$} As in the proof of (2), we find that
\begin{align*}
    \mathfrak{I}_\eta[f](p) - \mathfrak{I}_{\eta \oplus y}[f](p)
    \ge \frac{m}{2} (x_{k+1} - y)(p-x_k)\qquad\text{for $p\in(x_k,y)$},
\end{align*}
and
\begin{align*}
    \mathfrak{I}_\eta[f](p) - \mathfrak{I}_{\eta \oplus y}[f](p)
    \ge \frac{m}{2} (x_{k+1} - p)(y-x_k)\qquad\text{for $p\in(y,x_{k+1})$}.
\end{align*}
Putting the estimates together, we obtain
\begin{align*}
    J(p,\eta) - J(p,\eta{\oplus}y) &= \mathfrak{I}_\eta[f](p) - \mathfrak{I}_{\eta \oplus y}[f](p) \\
    &\ge \frac{m}{M}\biggl[\frac{x_{k+1} - y}{x_{k+1} - p} \mathbf{1}_{(x_k,y)}(p) + \frac{y-x_k}{p-x_k} \mathbf{1}_{(y,x_{k+1})}(p)\biggr]\bigl(\mathfrak{I}_\eta[f](p) - f(p)\bigr),
\end{align*}
which is point (3) of the assertion.
\end{proof}

We note that Lemma~\ref{lemma: interpolation error decreasing} directly implies that the global error function 
\[
    \scrG{:}\Omega\to [0,+\infty),\quad  \eta\mapsto \scrG(\eta) \coloneq \int_a^b \err(q, \eta) \, dq,
\]
satisfies $\scrG(\eta {\oplus} y) \le \scrG(\eta)$. 
%Our next lemma states that $\scrG$ satisfies Assumption~\ref{assumption: average error}\ref{item: improvement factor}. 

% \begin{lemma}
% \label{lemma: assumption item2}
%     Let $f{:} [a,b] \to \R$ be a continuous function. If $\scrG(\eta) = \scrG(\eta {\oplus} y)$ for $\calL\otimes \Pinf$-almost every $(y,\eta)$, then $\scrG(\eta) = 0$.
% \end{lemma}
% \begin{proof}
% Suppose there exists a $\eta \in \Omega$ with $\scrG(\eta) \neq 0$. We show that this implies $\scrG(\eta) > \scrG(\eta {\oplus} y)$ for $y \in B$ with $\calL(B) > 0$.
% Since $\scrG(\eta) > 0$, this implies there exists a Lebesgue-measurable set $B \subset (a,b)$ such that $\mathfrak{I}_\eta[f] > f$ on $B$. In particular, it holds that
% \begin{align*}
%     |\mathfrak{I}_\eta[f]-f|(y) > 0 = |\mathfrak{I}_{\eta \oplus y}[f] - f|(y)\qquad \text{for almost every $y\in B$.}
% \end{align*}
% Hence, due to continuity, there exists a $\delta > 0$ such that
% \begin{align*}
%     |\mathfrak{I}_\eta[f] - f|(q) > |\mathfrak{I}_{\eta \oplus y}[f] - f|(q) \qquad \text{for every $q \in B_\delta(y)$}.
% \end{align*}
% As a result of the monotonicity of the integral, we conclude $\scrG(\eta) > \scrG(\eta {\oplus} y)$. 
% \end{proof}

Before proving Lemma~\ref{lemma: assumptions satisfied}, we state one more lemma that allows us to conclude that $\scrG$ satisfies Assumption~\ref{assumption: average error}\ref{item: improvement factor}.

\begin{lemma}
\label{lemma: interpolation assumption3}
    Let $f\in C^2[a,b]$ be a strongly $m$-convex function with $
        0 < m \le f'' \le M$ on $[a,b]$. Further, let $\mu \in (0, 1/2)$, $\delta \coloneq (1-2\mu) > 0$, and $B_\mu(\eta)$ be given by~\eqref{eq: B(eta)}. If $\lambda(\eta, B_\mu(\eta)) \ge \delta$, then
    \[
        \int_{B_\mu(\eta)} \bigl[\scrG(\eta) - \scrG(\eta {\oplus} y)\bigr]\lambda(\eta, dy) \ge \frac{\delta \mu m}{M} \scrG(\eta)\qquad\text{for every $\eta\in\Omega$}.
    \]
\end{lemma}
% We note that a constant $m$ exists due to the strong convexity of $f$ and $M$ exists since $f''(x)$ is continuous on the closed interval $[a,b]$.

% We also note that
% \begin{align*}
%     \calL(B_\mu(\eta)) &= \sum_{k=0}^n (x_{k+1} - \mu(x_{k+1}-x_k) - x_k - \mu(x_{k+1} -x_k),\\
%     &= \sum_{k=0}^n (1-2\mu) (x_{k+1}-x_k) = (1-2\mu) (b-a) > 0.
% \end{align*}
\begin{proof}
    From Lemma~\ref{lemma: interpolation error decreasing}(3), we deduce that
    \begin{align*}
    &\int_{B_\mu(\eta)} \bigl[\scrG(\eta) - \scrG(\eta{\oplus}y)\bigr]\lambda(\eta,dy) = \iint_{[a,b]\times B_\mu(\eta)} \bigl[ J(q,\eta)- J(q,\eta{\oplus} y)\bigr]\lambda(\eta,dy)\,dq \\
    &\hspace{2em}\ge \frac{m}{M}\sum_{k=0}^{n-1}\iint_{[a,b]\times I_k^\mu} \biggl[\frac{x_{k+1} - y}{x_{k+1} - q} \mathbf{1}_{(q,x_{k+1})}(y) + \frac{y-x_k}{q-x_k} \mathbf{1}_{(x_k,q)}(y)\biggr]J(q,\eta) \lambda(\eta,dy)\,dq \\
    &\hspace{2em}\ge \mu\frac{m}{M}\sum_{k=0}^{n-1}\lambda(\eta,I_k^\mu)\int_{[a,b]} J(q,\eta)\,dq = \mu(1-2\mu)\frac{m}{M} \scrG(\eta),
    % &\ge \frac{m}{2}\sum_{k=0}^{n-1}\iint_{[a,b]\times I^\mu_k} (q-x_k)(x_{k+1}-y) \lambda(\eta,dy)\,dq \\
    % &\ge \mu\frac{m}{2}\sum_{k=0}^{n-1}\lambda(\eta,I^\mu_k) \int_{[a,b]} (q-x_k)(x_{k+1}-x_k) \,dq \\
    % &\ge \mu(1-2\mu)\frac{m}{M} \int_{[a,b]} J(q,\eta) \,dq = \mu(1-2\mu)\frac{m}{M}\scrG(\eta).
\end{align*}
where we used Fubini to interchange the order of the integral and the fact that
\[
    \frac{x_{k+1} - y}{x_{k+1} - q} \mathbf{1}_{(q,x_{k+1})}(y) + \frac{y-x_k}{q-x_k} \mathbf{1}_{(x_k,q)}(y) \ge \mu \mathbf{1}_{(x_k,x_{k+1})}(y)\qquad\text{for almost every $y\in I_k^\mu$}.\qedhere
\]
\end{proof}

\begin{proof}[Proof of Lemma~\ref{lemma: assumptions satisfied}]
We show that all the items in Assumption~\ref{assumption: average error} are satisfied.
\begin{enumerate}
    \item Let $x_0 = a$ and $x_1 = b$. Let
    \[\mathfrak{I}_0[f](x) = \frac{f(b)-f(a)}{b-a}(x-a) + f(a).\]
    Then for any $\eta_0 = (p, \phi, \ldots)$, Lemma~\ref{lemma: interpolation error decreasing} implies that
    \begin{align*}
        \scrG(\eta_0) \le \int_P |\mathfrak{I}_0[f](x) -f(x)| \le c_0,
    \end{align*}
    where
    \begin{align*}
        c_0 = (b-a)\cdot \max_{x \in [a,b]} |\mathfrak{I}_0[f](x) - f(x)|.
    \end{align*}
    \item This is a consequence of Lemma~\ref{lemma: interpolation error decreasing}.
    % \item This follows from Lemma~\ref{lemma: assumption item2} and the fact that both $\lambda_{\mathrm{unif}}(\eta, \cdot) \sim \calL$ and $\lambda_{{\rm rpdm}}(\eta, \cdot) \sim \calL$.
    \item[(3')] Let $\gamma = \frac{\mu m}{M} \in (0,1)$, let $\beta = 0$, let $\delta = 1-2\mu > 0$, and for every $\eta \in \Omega$, let $B_\mu(\eta)$ be given by~\eqref{eq: B(eta)}. Then this item follows from Lemma~\ref{lemma: interpolation assumption3}.
\end{enumerate}
\end{proof}

\begin{proof}[Proof of Theorem~\ref{theorem: interpolation C2}]
    Since $f \in C^2([a,b])$, setting $h_\alpha \coloneq \alpha \frac{|x|^2}{2}$, we have that the function
    \[
    x \mapsto f_\alpha(x)\coloneq f(x) + h_\alpha(x)\qquad\text{is strongly convex for $\alpha > c_f$},
    \]
    since 
    \[m_{f_\alpha} \coloneq \alpha - c_f \le f''(x) \le \alpha + c_f \coloneq M_{f_\alpha},\]
    for all $x \in [a,b]$. Furthermore, $h_\alpha$ is strongly convex, since $h_\alpha''(x) = \alpha$ for all $x \in [a,b]$.
    Moreover,
    \begin{align*}
        \|\mathfrak{I}_{\eta_t}[f] - f\|_{L_1} &= \|\mathfrak{I}_{\eta_t}[f_\alpha - h_\alpha] - (f_\alpha - h_\alpha)\|_{L_1}\\
        &\le \|\mathfrak{I}_{\eta_t}[f_\alpha] - f_\alpha\|_{L_1} + \|\mathfrak{I}_{\eta_t}[h_\alpha] - h_\alpha\|_{L_1}.
    \end{align*}
    Let $\mu \in (0, 1/2)$, and let $\eps > 0$. 
    %We note that the result $\lambda_{\mathrm{rpdm}}(\eta, B_\mu(\eta)) = \lambda_{\mathrm{unif}}(\eta, B_\mu(\eta)) = \delta$ with $\delta \coloneq 1-2\mu$ is independent of the function $f$. Therefore, 
    We can follow the proof of Lemma~\ref{lemma: assumptions satisfied}, and use the result of Theorem~\ref{thm: convergence rate average} to conclude
    \begin{align*}
        \p\Big(\|\mathfrak{I}_{\eta_t}[f_\alpha] - f_\alpha\|_{L_1} > \frac{\eps}{2}\Big) \le \frac{2}{\eps} \|\mathfrak{I}_{\eta_0}[f_\alpha] - f_\alpha\|_{L_1} e^{-\gamma_{f_\alpha} \delta t}, 
    \end{align*}
    with $\gamma_{f_\alpha} \coloneq \mu \frac{\alpha - c_f}{\alpha + c_f}$, and 
    \begin{align*}
        \p\Big(\|\mathfrak{I}_{\eta_t}[h_\alpha] - h_\alpha\|_{L_1} > \frac{\eps}{2}\Big) \le \frac{2}{\eps} \|\mathfrak{I}_{\eta_0}[h_\alpha] - h_\alpha\|_{L_1} e^{-\gamma_{h_\alpha} \delta t},
    \end{align*}
    with $\gamma_{h_\alpha} = \mu$. 
    
    We note that $\gamma_{f_\alpha} \le \gamma_{h_\alpha}$ for every $\alpha > c_f$. Let $m_\scrG \coloneq 4 \max\{\|\mathfrak{I}_{\eta_0}[f_\alpha] - f_\alpha\|_{L_1}, \|\mathfrak{I}_{\eta_0}[h_\alpha] - h_\alpha\|_{L_1}\}$, then it holds that
    \begin{align*}
        \p\Big(\|\mathfrak{I}_{\eta_t}[f] - f\|_{L_1} > \eps\Big) &= \p\Big(\|\mathfrak{I}_{\eta_t}[f_\alpha] - f_\alpha\|_{L_1} > \frac{\eps}{2}\Big)  + \p\Big(\|\mathfrak{I}_{\eta_t}[h_\alpha] - h_\alpha\|_{L_1} > \frac{\eps}{2}\Big)\\
        &\le \frac{m_\scrG}{\eps} e^{-\gamma_{f_\alpha} \delta t} \quad \text{for} \, t \ge 1. 
    \end{align*}
    This concludes the proof.
\end{proof}

%The proof of Lemma~\ref{lemma: assumptions satisfied} can be found in Appendix~\ref{app: interpolation}. 
%The result holds for $\lambda_{\mathrm{unif}}$ and $\lambda_{\mathrm{rpdm}}$. In numerical experiments, we also compute the error and variance for a classical weak greedy method (see~\eqref{eq: weak greedy}). %but we must note that the Saturation Property in Assumption~\ref{assumption: average error} does not hold for this method.

% \begin{remark}
%     The weak greedy method depends on a sample set $S \subset P$. A consequence of this formulation is that the Saturation Property of Assumption~\ref{assumption: average error} is not satisfied in this application. It is possible to reformulate Assumption~\ref{assumption: average error} and Theorem~\ref{thm: convergence rate average} by replacing all instances of $P$ by the discrete set $S$. Similarly, the error function~\eqref{eq: error interpolation} becomes
%     \[
%         \scrG_S(\eta) = \sum_{x \in S} J(x,\eta).
%     \]
%     In this reformulated framework, we can expect convergence results of $\scrG_S$ for the weak greedy method in the interpolation setting.
% \end{remark}

\section{Randomized Polytope Division Method}
\label{sec: rpdm}
%The convergence results of Lemma~\ref{lemma: convergence J} and Lemma~\ref{lemma: average convergence} hold for general rate functions $\lambda$ that satisfy Assumption~\ref{assumption: rate function}, and Theorem~\ref{thm: convergence rate average} holds if Assumption~\ref{assumption: average error} is also satisfied.
%So far, our results have been formulated for a general transition kernel $\lambda{:} \Omega \times P\to [0,+\infty)$ satisfying Assumption~\ref{assumption: rate function} or Assumption~\ref{assumption: average error}.
Several of the results in this work depend on general assumptions on the transition kernel $\lambda$.
In this section, we consider specific choices for this transition kernel and show these kernels satisfy the assumptions. In particular, we consider the transition kernel corresponding to a randomized version of the Polytope Division Method (R-PDM), and a uniform transition kernel. We first describe R-PDM.

\subsection{Randomized Polytope Division Method}
R-PDM is an algorithm that divides a \emph{hyperrectangle} parameter set $P$ into polytopes and searches for regions where the local error $J$ is large. Each configuration $\eta = (\eta_1,\ldots, \eta_n, \phi, \ldots)$ corresponds to a specific set $\calD(\eta)$ of polytopes that divide the parameter set $P$ (detailed construction of $\calD(\eta)$ can be found below). The transition kernel corresponding to R-PDM is given by
\begin{align}
\label{eq: rate function rpdm}
\lambda_{\rm{rpdm}}(\eta, dy) \coloneq \sum_{D \in \calD(\eta)} \frac{e^{\alpha \,  \dashintfrac_D \err(q, \eta) \, dq}}{Z} \mathrm{unif}_D(dy),
\end{align}
where $\alpha >0$, $Z$ is a normalization factor and $\mathrm{unif}_D$ is the uniform measure on the polytope $D$.

This transition kernel is purely theoretical because, in practice, computing $\dashint_D \err(q, \eta) \, dq$ exactly is often infeasible. Therefore, we approximate this integral in practical applications by 
\begin{align}
\label{eq: average error approximation}
    \dashint_D \err(q, \eta) \, dq\approx \err(p_D, \eta),
\end{align}
where $p_D \in D$. As an initial approximation, let $p_D \coloneq \sfb_D$, with $\sfb_D$ being the barycenter of $D$. Since the barycenter could lead to a bad approximation of the integral $\dashint_{D} \err(q, \eta) \, dq$, we can replace $\sfb_D$ by an arbitrary point $p_D$ whenever $\err(\sfb_D, \eta) < \epsilon$ for some preset tolerance $\epsilon$.

\begin{remark}
    The approximation given by~\eqref{eq: average error approximation} with $p_D=\sfb_D$ is the midpoint rule with one quadrature point. The approximation can be improved by either using more quadrature points or using a stochastic integrator at the expense of computational efficiency and scalability.
\end{remark}

The transition kernel $\lambda$ depends on the polytope division $\calD(\eta)$. In R-PDM, this division is refined after $\eta$ transitions to a new state $\eta{\oplus} y$. This refinement is based on an operation called \emph{facet linking}, which is defined as in \cite{nielen2025polytope} by
\begin{definition}
    Let $P \subset \R^d$ be a polytope and $p \in P$  be an arbitrary point. Furthermore, let $\partial P$ denote the set of facets of $P$.
    %Every facet $F_i$ is the convex hull of a set of vertices $\{v_1, \ldots, v_{L_i}\}$ for all $i = 1, \ldots, M$. 
    Then the \emph{facet linking operator} $\FL$ is given by
    \begin{align*}
        \FL(p, P) = \big\{\conv(p\cup F) \, :\, F\in\partial P\big\}.
    \end{align*}
\end{definition}
In other words, the facet linking operator divides a polytope by connecting a point $p$ to all facets $F \in \partial P$.
The polytope division $\calD(\eta)$ depends on the facet linking operator in the following way: Suppose $\eta$ is a state in R-PDM with polytope division $\calD(\eta)$, and $\eta \mapsto \eta {\oplus} y$ for some $y \in P$ with $y \in D \in \calD(\eta)$, then $\calD(\eta {\oplus} y) = \calD(\eta) \backslash D {\cup} \FL(y, D)$. Figure~\ref{fig: PDM2d} displays an example of the first steps of R-PDM in a $2$-dimensional case. The method is summarized in the following algorithm:
\begin{algorithm}
\caption{Randomized Polytope Division Method}\label{rpdm}
%\algorithmicrequire{Hypercube $\calP$, stopping tolerance $\varepsilon$, starting point $\mu^1 \in \calP$}\\
%\algorithmicensure{Basis $V_\RB \in \R^{N_h \times M}$}
\begin{algorithmic}[1]
\State Initialize number of points $n$, constant $\alpha >0$, and tolerance $\epsilon > 0$
\State $k\gets 1$
\State Choose $p \in \interior{P}$
\State Set $\eta \coloneq (p, \phi, \ldots)$
% \State Set $\calD_1 = \text{FL}(\mu_1, \calP)$
\State Set $\calD\coloneq\{D\in \FL(p, P)\}$
\State $p_D = \sfb_D$ for all $D \in \calD$
\While{$k \, < \, n$}
\State Compute $\err(p_D, \eta)$ for all $D \in \calD$.
\State Sample $D \in \calD(\eta)$ with probability weighted by $e^{\alpha \err(p_D, \, \eta)}$
\State Sample $y \in D$ according to $\mathrm{unif}_D$
\State $\eta \gets \eta {\oplus} y$
\State Resample $p_E \in E$ uniformly in $E$ for all $E \in \calD(\eta) \backslash D$ with $\err(p_E, \eta) < \epsilon$
\State $\calD \gets (\calD \backslash \{D\}) \cup \FL(y, D)$
\State For all $E \in \FL(y, D)$, set $p_E = \sfb_E$
\State $k \gets k+1$
\EndWhile
\end{algorithmic}
\end{algorithm}

\begin{figure}[h]
\centering
    \begin{subfigure}{0.24\linewidth}
        \centering
        \begin{tikzpicture}
        \def\d{2};
        \def\o{2};

        % coarse
        \coordinate (S1bl) at (0,0);
        \coordinate (S1br) at (\d,0);
        \coordinate (S1tr) at (\d,\d);
        \coordinate (S1tl) at (0,\d);
        \coordinate (S1c) at (\d/2,\d/2);
        \coordinate (bc1) at (\d/4, \d/2);
        \coordinate (bc2) at (\d/2, \d/4);
        \coordinate (bc3) at (3*\d/4, \d/2);
        \coordinate (bc4) at (\d/2, 3*\d/4);

        %\draw[draw=none, fill=red] (S1bl) circle (2pt);
        %\draw[draw=none, fill=red] (S1br) circle (2pt);
        %\draw[draw=none, fill=red] (S1tr) circle (2pt);
        %\draw[draw=none, fill=red] (S1tl) circle (2pt);

        \draw[thick, color=black] (S1bl) -- (S1br) -- (S1tr) -- (S1tl) -- (S1bl);
        \draw[thick, color=red] (S1bl) -- (S1c) -- (S1br);
        \draw[thick, color=red] (S1tl) -- (S1c) -- (S1tr);
        \draw[draw=none, fill=red] (S1c) circle (2pt);

        \coordinate (S1trc) at (3*\d/4,3*\d/4);
        \coordinate (S1tc) at (\d/2,\d);
        \coordinate (S1tlc) at (\d/4,3*\d/4);

\end{tikzpicture}
        \caption*{Step (1)}
        \label{fig: sf1: PDM2d}
    \end{subfigure}
    %\hspace{.1cm}
    %\hfill
    \begin{subfigure}{0.24\linewidth}
        \centering
        \begin{tikzpicture}
        \def\d{2};
        \def\o{2};

        % coarse
        \coordinate (S1bl) at (0,0);
        \coordinate (S1br) at (\d,0);
        \coordinate (S1tr) at (\d,\d);
        \coordinate (S1tl) at (0,\d);
        \coordinate (S1c) at (\d/2,\d/2);
        \coordinate (bc1) at (\d/6, \d/2);
        \coordinate (bc2) at (\d/2, \d/6);
        \coordinate (bc3) at (5*\d/6, \d/2);
        \coordinate (bc4) at (\d/2, 5*\d/6);

        %\draw[draw=none, fill=red] (S1bl) circle (2pt);
        %\draw[draw=none, fill=red] (S1br) circle (2pt);
        %\draw[draw=none, fill=red] (S1tr) circle (2pt);
        %\draw[draw=none, fill=red] (S1tl) circle (2pt);

        \draw[draw=none, fill = blue] plot[mark=x] coordinates{(bc1)};
        \draw[draw=none, fill = blue] plot[mark=x] coordinates{(bc2)};
        \draw[draw=none, fill = blue] plot[mark=x] coordinates{(bc3)};
        \draw[draw=none, fill = blue] plot[mark=x] coordinates{(bc4)};

        \draw[thick, color=black] (S1bl) -- (S1br) -- (S1tr) -- (S1tl) -- (S1bl);
        \draw[thick, color=red] (S1bl) -- (S1c) -- (S1br);
        \draw[thick, color=red] (S1tl) -- (S1c) -- (S1tr);
        \draw[draw=none, fill=red] (S1c) circle (2pt);

        \coordinate (S1trc) at (3*\d/4,3*\d/4);
        \coordinate (S1tc) at (\d/2,\d);
        \coordinate (S1tlc) at (\d/4,3*\d/4);

\end{tikzpicture}
        \caption*{Step (2)}
        \label{fig: sf2: PDM2d}
    \end{subfigure}
    %\hspace{.1cm}
        %\hfill
    \begin{subfigure}{0.24\linewidth}
        \centering
        \begin{tikzpicture}
        \def\d{2};
        \def\o{2};

        % coarse
        \coordinate (S1bl) at (0,0);
        \coordinate (S1br) at (\d,0);
        \coordinate (S1tr) at (\d,\d);
        \coordinate (S1tl) at (0,\d);
        \coordinate (S1c) at (\d/2,\d/2);
        \coordinate (bc1) at (\d/7, 3*\d/5);
        \coordinate (bc2) at (\d/3, 3*\d/14);
        \coordinate (bc3) at (5*\d/6, \d/2);
        \coordinate (bc4) at (\d/2+\d/10, 7*\d/8);

        %\draw[draw=none, fill=red] (S1bl) circle (2pt);
        %\draw[draw=none, fill=red] (S1br) circle (2pt);
        %\draw[draw=none, fill=red] (S1tr) circle (2pt);
        %\draw[draw=none, fill=red] (S1tl) circle (2pt);

        %\draw[draw=none, fill = blue] plot[mark=x] coordinates{(bc1)};
        %\draw[draw=violet, fill = violet] plot[mark=x] coordinates{(bc2)};
        %\draw[draw=none, fill = blue] plot[mark=x] coordinates{(bc3)};
        %\draw[draw=violet, fill = violet] plot[mark=x] coordinates{(bc4)};

        \draw[thick, color=black] (S1bl) -- (S1br) -- (S1tr) -- (S1tl) -- (S1bl);
        \draw[thick, color=red] (S1bl) -- (S1c) -- (S1br);
        \draw[thick, color=red] (S1tl) -- (S1c) -- (S1tr);
        %\draw[thick, color=red] (S1bl) -- (bc1) -- (S1c);
        %\draw[thick, color=red] (bc1) -- (S1tl);
        \draw[draw=none, fill=red] (S1c) circle (1pt);
        \draw[draw=none, fill=red] (bc1) circle (1pt);

        \coordinate (S1trc) at (3*\d/4,3*\d/4);
        \coordinate (S1tc) at (\d/2,\d);
        \coordinate (S1tlc) at (\d/4,3*\d/4);

\end{tikzpicture}
        \caption*{Step (3)}
        \label{fig: sf3: PDM2d}
    \end{subfigure}
    %\hspace{.1cm}
        %\hfill
    \begin{subfigure}{0.24\linewidth}
        \centering
        \begin{tikzpicture}
        \def\d{2};
        \def\o{2};

        % coarse
        \coordinate (S1bl) at (0,0);
        \coordinate (S1br) at (\d,0);
        \coordinate (S1tr) at (\d,\d);
        \coordinate (S1tl) at (0,\d);
        \coordinate (S1c) at (\d/2,\d/2);
        \coordinate (bc1) at (\d/6, \d/2);
        \coordinate (bc2) at (\d/2, \d/6);
        \coordinate (bc3) at (5*\d/6, \d/2);
        \coordinate (bc4) at (\d/2, 5*\d/6);

        %\draw[draw=none, fill=red] (S1bl) circle (2pt);
        %\draw[draw=none, fill=red] (S1br) circle (2pt);
        %\draw[draw=none, fill=red] (S1tr) circle (2pt);
        %\draw[draw=none, fill=red] (S1tl) circle (2pt);

        %\draw[draw=none, fill = blue] plot[mark=x] coordinates{(bc1)};
        %\draw[draw=none, fill = blue] plot[mark=x] coordinates{(bc2)};
        %\draw[draw=none, fill = blue] plot[mark=x] coordinates{(bc3)};
        %\draw[draw=none, fill = blue] plot[mark=x] coordinates{(bc4)};

        \draw[draw = none, fill = blue, opacity=0.45] (S1bl) -- (S1c) -- (S1tl);
        
        \draw[thick, color=black] (S1bl) -- (S1br) -- (S1tr) -- (S1tl) -- (S1bl);
        \draw[thick, color=red] (S1bl) -- (S1c) -- (S1br);
        \draw[thick, color=red] (S1tl) -- (S1c) -- (S1tr);
        \draw[draw=none, fill=red] (S1c) circle (2pt);
        %\draw[draw=blue, fill= none] (bc1) circle (4pt);

        \coordinate (S1trc) at (3*\d/4,3*\d/4);
        \coordinate (S1tc) at (\d/2,\d);
        \coordinate (S1tlc) at (\d/4,3*\d/4);

\end{tikzpicture}
        \caption*{Step (4)}
        \label{fig: sf4: PDM2d}
    \end{subfigure}
    %\hspace{.1cm}
        %\hfill
    \begin{subfigure}{0.24\linewidth}
        \centering
        \begin{tikzpicture}
        \def\d{2};
        \def\o{2};

        % coarse
        \coordinate (S1bl) at (0,0);
        \coordinate (S1br) at (\d,0);
        \coordinate (S1tr) at (\d,\d);
        \coordinate (S1tl) at (0,\d);
        \coordinate (S1c) at (\d/2,\d/2);
        \coordinate (bc1) at (\d/6, \d/2);
        \coordinate (bc2) at (\d/2, \d/6);
        \coordinate (bc3) at (5*\d/6, \d/2);
        \coordinate (bc4) at (\d/2, 5*\d/6);

        %\draw[draw=none, fill=red] (S1bl) circle (2pt);
        %\draw[draw=none, fill=red] (S1br) circle (2pt);
        %\draw[draw=none, fill=red] (S1tr) circle (2pt);
        %\draw[draw=none, fill=red] (S1tl) circle (2pt);

        \draw[draw=none, fill = blue] plot[mark=x] coordinates{(bc1)};
        \draw[draw=none, fill = blue] plot[mark=x] coordinates{(bc2)};
        \draw[draw=none, fill = blue] plot[mark=x] coordinates{(bc3)};
        \draw[draw=none, fill = blue] plot[mark=x] coordinates{(bc4)};

        \draw[thick, color=black] (S1bl) -- (S1br) -- (S1tr) -- (S1tl) -- (S1bl);
        \draw[thick, color=red] (S1bl) -- (S1c) -- (S1br);
        \draw[thick, color=red] (S1tl) -- (S1c) -- (S1tr);
        \draw[draw=none, fill=red] (S1c) circle (2pt);
        %\draw[draw=blue, fill= none] (bc1) circle (4pt);
        \draw[draw=violet, fill= none] (bc4) circle (4pt);
        \draw[draw=violet, fill= none] (bc2) circle (4pt);

        \coordinate (S1trc) at (3*\d/4,3*\d/4);
        \coordinate (S1tc) at (\d/2,\d);
        \coordinate (S1tlc) at (\d/4,3*\d/4);

\end{tikzpicture}
        \caption*{Step (5)}
        \label{fig: sf5: PDM2d}
    \end{subfigure}
    %    \hspace{.1cm}
        %\hfill
    \begin{subfigure}{0.24\linewidth}
        \centering
        \begin{tikzpicture}
        \def\d{2};
        \def\o{2};

        % coarse
        \coordinate (S1bl) at (0,0);
        \coordinate (S1br) at (\d,0);
        \coordinate (S1tr) at (\d,\d);
        \coordinate (S1tl) at (0,\d);
        \coordinate (S1c) at (\d/2,\d/2);
        \coordinate (bc1) at (\d/7, 3*\d/5);
        \coordinate (bc2) at (\d/3, 3*\d/14);
        \coordinate (bc3) at (5*\d/6, \d/2);
        \coordinate (bc4) at (\d/2+\d/10, 7*\d/8);

        %\draw[draw=none, fill=red] (S1bl) circle (2pt);
        %\draw[draw=none, fill=red] (S1br) circle (2pt);
        %\draw[draw=none, fill=red] (S1tr) circle (2pt);
        %\draw[draw=none, fill=red] (S1tl) circle (2pt);

        %\draw[draw=none, fill = blue] plot[mark=x] coordinates{(bc1)};
        \draw[draw=violet, fill = violet] plot[mark=x] coordinates{(bc2)};
        %\draw[draw=none, fill = blue] plot[mark=x] coordinates{(bc3)};
        \draw[draw=violet, fill = violet] plot[mark=x] coordinates{(bc4)};

        \draw[thick, color=black] (S1bl) -- (S1br) -- (S1tr) -- (S1tl) -- (S1bl);
        \draw[thick, color=red] (S1bl) -- (S1c) -- (S1br);
        \draw[thick, color=red] (S1tl) -- (S1c) -- (S1tr);
        %\draw[thick, color=red] (S1bl) -- (bc1) -- (S1c);
        %\draw[thick, color=red] (bc1) -- (S1tl);
        \draw[draw=none, fill=red] (S1c) circle (1pt);
        %\draw[draw=none, fill=red] (bc1) circle (1pt);

        \coordinate (S1trc) at (3*\d/4,3*\d/4);
        \coordinate (S1tc) at (\d/2,\d);
        \coordinate (S1tlc) at (\d/4,3*\d/4);

\end{tikzpicture}
        \caption*{Step (6)}
        \label{fig: sf6: PDM2d}
    \end{subfigure}
    %    \hspace{.1cm}
        %\hfill
    \begin{subfigure}{0.24\linewidth}
        \centering
        \begin{tikzpicture}
        \def\d{2};
        \def\o{2};

        % coarse
        \coordinate (S1bl) at (0,0);
        \coordinate (S1br) at (\d,0);
        \coordinate (S1tr) at (\d,\d);
        \coordinate (S1tl) at (0,\d);
        \coordinate (S1c) at (\d/2,\d/2);
        \coordinate (bc1) at (\d/7, 3*\d/5);
        \coordinate (bc2) at (\d/2, \d/6);
        \coordinate (bc3) at (5*\d/6, \d/2);
        \coordinate (bc4) at (\d/2, 5*\d/6);

        %\draw[draw=none, fill=red] (S1bl) circle (2pt);
        %\draw[draw=none, fill=red] (S1br) circle (2pt);
        %\draw[draw=none, fill=red] (S1tr) circle (2pt);
        %\draw[draw=none, fill=red] (S1tl) circle (2pt);

        %\draw[draw=none, fill = blue] plot[mark=x] coordinates{(bc1)};
        %\draw[draw=none, fill = blue] plot[mark=x] coordinates{(bc2)};
        %\draw[draw=none, fill = blue] plot[mark=x] coordinates{(bc3)};
        %\draw[draw=none, fill = blue] plot[mark=x] coordinates{(bc4)};

        \draw[thick, color=black] (S1bl) -- (S1br) -- (S1tr) -- (S1tl) -- (S1bl);
        \draw[thick, color=red] (S1bl) -- (S1c) -- (S1br);
        \draw[thick, color=red] (S1tl) -- (S1c) -- (S1tr);
        \draw[thick, color=red] (S1bl) -- (bc1) -- (S1c);
        \draw[thick, color=red] (bc1) -- (S1tl);
        \draw[draw=none, fill=red] (S1c) circle (1pt);
        \draw[draw=none, fill=red] (bc1) circle (1pt);

        \coordinate (S1trc) at (3*\d/4,3*\d/4);
        \coordinate (S1tc) at (\d/2,\d);
        \coordinate (S1tlc) at (\d/4,3*\d/4);

\end{tikzpicture}
        \caption*{Step (7)}
        \label{fig: sf7: PDM2d}
    \end{subfigure}
    %    \hspace{.1cm}
        %\hfill
    \begin{subfigure}{0.24\linewidth}
        \centering
        \begin{tikzpicture}
        \def\d{2};
        \def\o{2};

        % coarse
        \coordinate (S1bl) at (0,0);
        \coordinate (S1br) at (\d,0);
        \coordinate (S1tr) at (\d,\d);
        \coordinate (S1tl) at (0,\d);
        \coordinate (S1c) at (\d/2,\d/2);
        \coordinate (bc1) at (\d/7, 3*\d/5);
        \coordinate (bc2) at (\d/3, 3*\d/14);
        \coordinate (bc3) at (5*\d/6, \d/2);
        \coordinate (bc4) at (\d/2+\d/10, 7*\d/8);

        %\draw[draw=none, fill=red] (S1bl) circle (2pt);
        %\draw[draw=none, fill=red] (S1br) circle (2pt);
        %\draw[draw=none, fill=red] (S1tr) circle (2pt);
        %\draw[draw=none, fill=red] (S1tl) circle (2pt);

        \draw[thick, color=black] (S1bl) -- (S1br) -- (S1tr) -- (S1tl) -- (S1bl);
        \draw[thick, color=red] (S1bl) -- (S1c) -- (S1br);
        \draw[thick, color=red] (S1tl) -- (S1c) -- (S1tr);
        \draw[thick, color=red] (S1bl) -- (bc1) -- (S1c);
        \draw[thick, color=red] (bc1) -- (S1tl);
        \draw[draw=none, fill=red] (S1c) circle (1pt);
        \draw[draw=none, fill=red] (bc1) circle (1pt);

        \coordinate (bc11) at (\d/18,\d/2);
        \coordinate (bc12) at (\d/4,\d/3);
        \coordinate (bc13) at (\d/4,2*\d/3);

        \draw[draw=black, fill=none] plot[mark=x] coordinates{(bc11)};
        \draw[draw=black, fill=none] plot[mark=x] coordinates{(bc12)};
        \draw[draw=black, fill=none] plot[mark=x] coordinates{(bc13)};
        \draw[draw=violet, fill=none] plot[mark=x] coordinates{(bc2)};
        \draw[draw=black, fill=none] plot[mark=x] coordinates{(bc3)};
        \draw[draw=violet, fill=none] plot[mark=x] coordinates{(bc4)};

\end{tikzpicture}
        \caption*{Step (8)}
        \label{fig: sf8: PDM2d}
    \end{subfigure}
\caption{Depiction of the steps in R-PDM for the $2$-dimensional parameter case and split domain via facet linking. (1) Sample first parameter and divide $P$ via facet linking (2) Compute barycenters. (3) Select a point based on the transition kernel (4) Mark the polytope containing this point (5) Select parameters in other polytope with error function below tolerance (6) Sample new points to replace these barycenters. (7) Update polytope division. (8) Compute the new barycenters.}
\label{fig: PDM2d}
\end{figure}
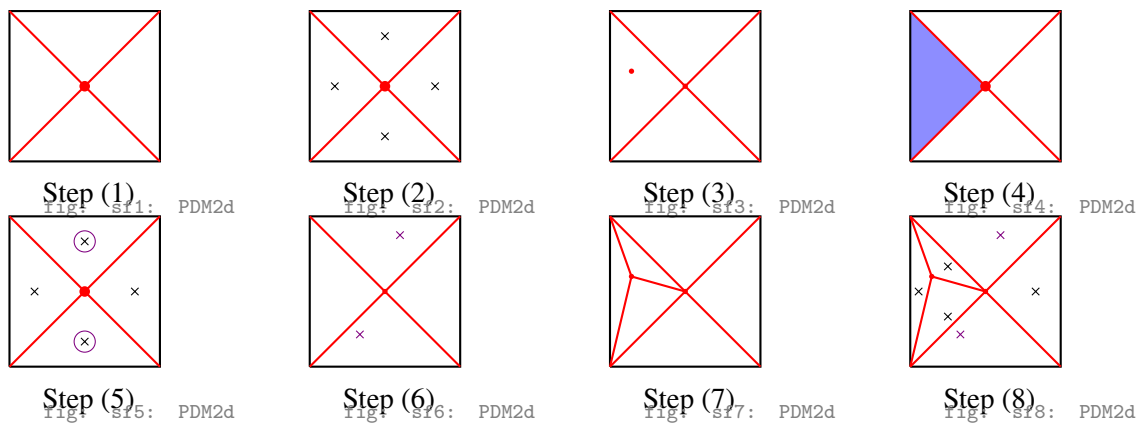

The idea behind R-PDM is to place more mass on regions of $P$ where the error $\err$ is higher. The transition kernel $\lambda_{\rm{rpdm}}$ satisfies Assumption~\ref{assumption: rate function}. Moreover, $\lambda_{\rm{rpdm}}(\eta, \cdot)$ is equivalent to the Lebesgue measure (i.e., $\lambda_{\rm{rpdm}}(\eta, \cdot) \sim \calL$), implying that the convergence results of Theorem~\ref{theorem: convergence J} and Theorem~\ref{theorem: average convergence} hold $\calL(dy) {\otimes} P_\infty$-almost everywhere. For applications where the global error function $\scrG$ satisfies Assumption~\ref{assumption: average error}, we establish convergence rates via Theorem~\ref{thm: convergence rate average}. More specifically, Theorem~\ref{theorem: interpolation C2} guarantees convergence for the interpolation of $C^2$ functions if $\lambda$ satisfies Assumption~\ref{assumption: interpolation transition kernel}. In Lemma~\ref{lemma: mass B_eta} below,  we show that this is indeed the case for $\lambda_{\mathrm{rpdm}}$.

However, there exist other transition kernels that are equivalent to the Lebesgue measure that satisfy the same property. One simple example is the uniform measure \[\lambda_{\mathrm{unif}}(\eta, \cdot) \coloneq \mathrm{unif}_P.\]
The selection of points via a uniform measure is computationally more efficient than R-PDM. However, computational efficiency is not the only factor to consider.

First, we must note that the estimates in the proof of Theorem~\ref{thm: convergence rate average} are crude. We consider estimates that hold uniformly in time and, therefore, disregard potential differences in local-in-time improvements. Indeed, in Assumption~\ref{assumption: average error}, we assume that $\lambda(\eta, B(\eta)) \ge \delta$ holds uniformly in time for some $\delta > 0$, leading to a bound that cannot be expected to be tight for every transition kernel as the measure of $B(\eta)$ can significantly exceed $\delta$ for certain $\eta$. Since R-PDM places more mass on regions with high error, we can locally expect this assumption to hold for a larger set $B(\eta)$ than $\lambda_\mathrm{unif}$.

Secondly, the convergence results hold with high probability, but in practice there might be a large difference in variance. A high variance indicates that, even though the error converges to zero in the expected form, single runs of the algorithm could lead to poor approximated solutions to the COP. A user might prefer an algorithm with lower variance to have more confidence in the convergence results of individual runs. At this moment, none of these additional factors appear in the theoretical results. We present numerical tests to show their importance and compute the variance of several greedy-type algorithms. 

%  We first define a set $B_\mu(\eta)$ as follows
% \begin{definition}
% \label{def: B(eta)}
% Let $\mu \in (0, 1/2)$, $\eta \in \Omega_n$, $x_0 = a, x_{n+1} = b$, and $\mfS(\eta) = (x_1, \ldots, x_n, \phi, \ldots)$. For each $k\in\{0,\ldots,n\}$, we set $I_k^\mu \coloneq (x_k + \mu(x_{k+1}-x_k), x_{k+1} - \mu (x_{k+1}-x_k))$. Then,
% \begin{align}
% \label{eq: B(eta)}
% B_\mu(\eta) \coloneq \bigl\{x \in [a,b]: x \in I^\mu_k,\; k \in \{0, \ldots, n\}\bigr\}.
% \end{align}
% \end{definition}

\medskip
We now show that $\lambda_{\mathrm{rpdm}}$ and $\lambda_{\mathrm{unif}}$ indeed satisfy Property~\ref{assumption: interpolation transition kernel}, which implies the exponential convergence result of Theorem~\ref{theorem: interpolation C2}.
\begin{lemma}
    \label{lemma: mass B_eta}
    Let $\eta \in \Omega$, let $\mu \in (0, 1/2)$, and $B_\mu(\eta)$ be given by~\eqref{eq: B(eta)}. Then Property~\ref{assumption: interpolation transition kernel} holds:
    \begin{align*}
        \lambda_{\mathrm{unif}}(\eta, B_\mu(\eta)) = \lambda_{{\rm rpdm}}(\eta, B_\mu(\eta)) = 1-2\mu > 0.
    \end{align*}
\end{lemma}
\begin{proof}
We note that
\begin{align*}
    \lambda_{\mathrm{unif}}(\eta, B_\mu(\eta)) = \sum_{k = 0}^n \int_{I^\mu_k} \frac{1}{b-a} dy &= \sum_{k=0}^n \frac{|I^\mu_k|}{b-a} = (1-2\mu) \sum_{k=0}^n \frac{x_{k+1}-x_k}{b-a} = 1-2\mu > 0. 
\end{align*}
Similarly, for R-PDM, we note that by construction for every $D \in \calD(\eta)$ there exists precisely one $k \in \{0, \ldots, n\}$ such that $I^\mu_k \subset D$. We also denote this $D$ by $D_k$.
Therefore,
\begin{align*}
    \lambda_{\mathrm{rpdm}}(\eta, B_\mu(\eta)) &= \sum_{k=0}^n \int_{I^\mu_k} \sum_{D \in \calD(\eta)} \frac{e^{\alpha \dashintfrac_D \err(z, \eta) \, dz}}{Z} \mathrm{unif}_D(dy) \\
    &= \sum_{k=0}^n \frac{e^{\alpha \dashintfrac_{D_k} \err(z, \eta) \, dz}}{Z} \frac{|I^\mu_k|}{x_{k+1} - x_k} = (1-2\mu)\sum_{k=0}^n \frac{e^{\alpha \dashintfrac_{D_k} \err(z, \eta) \, dz}}{Z} = 1-2\mu>0,
\end{align*}
which concludes the proof.
\end{proof}

The convergence results hold for $\lambda_{\mathrm{rpdm}}$ and $\lambda_{\mathrm{unif}}$. In the numerical experiments, we also compute the error and variance for a classical weak greedy method (see~\eqref{eq: weak greedy}),
\begin{remark}
    The weak greedy method depends on a sample set $S \subset P$. A consequence of this formulation is that the Saturation Property of Assumption~\ref{assumption: average error} is not satisfied in this application. It is possible to reformulate Assumption~\ref{assumption: average error} and Theorem~\ref{thm: convergence rate average} by replacing all instances of $P$ by the discrete set $S$. Similarly, the error function~\eqref{eq: error interpolation} becomes
    \[
        \scrG_S(\eta) = \sum_{x \in S} J(x,\eta).
    \]
    In this reformulated framework, we can expect convergence results of $\scrG_S$ for the weak greedy method in the interpolation setting.
\end{remark}

\subsection{Numerical experiments}
\subsubsection{Example 1}
We run the stochastic greedy methods to compute interpolation nodes for three different functions. \footnote{The collection of all the codes used to generate the numerical results presented in this subsection can be found here: https://gitlab.tue.nl/s158446/interpolation.git} The first function $f{:} [0,5] \to [0,+\infty)$ is given by $f(x) = x^2 + \frac{1}{30} x^4$ and depicted in Figure~\ref{fig: increasing-exact function}. We note that this function is infinitely many times differentiable and strongly convex, so Assumption~\ref{assumption: average error} is satisfied due to Lemma~\ref{lemma: assumptions satisfied}. For R-PDM, we set $\alpha = 500$ and tolerance $\epsilon = 0.01$. 

\begin{figure}[h]
    \begin{subfigure}{.49\textwidth}
    \centering
    \includegraphics[width=\linewidth]{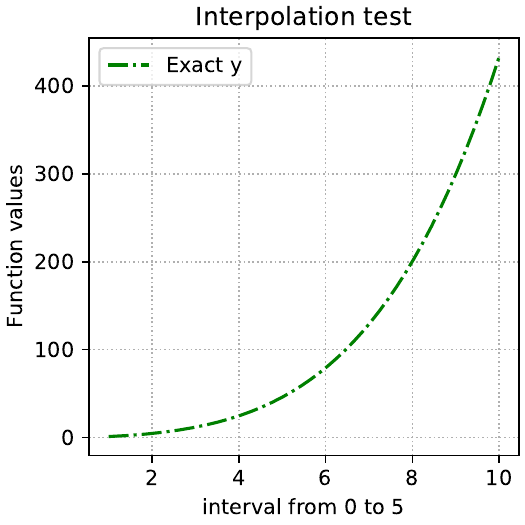}
    \caption{Exact function $f(x) = x^2 + \tfrac{1}{30}x^4$}
    \label{fig: increasing-exact function}
    \end{subfigure}
    \begin{subfigure}{.49\textwidth}
    \centering
    \includegraphics[width=\linewidth]{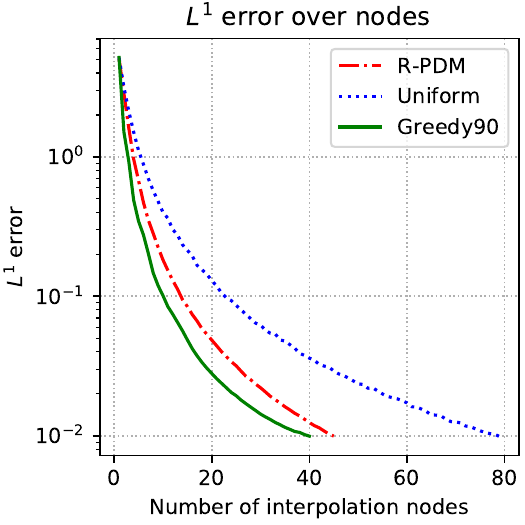}
    \caption{Decay of $L^1$-error over interpolation nodes}
    \label{fig: increasing-error decay}
    \end{subfigure}
    % \begin{minipage}[b]{0.49\linewidth}
        % \centering  
  % \caption{$f(x) = x^2 + x^4/30$}
  % \label{fig: increasing-exact function}
    % \end{minipage}
    % \hfill
    % \begin{minipage}[b]{0.49\linewidth}
  % \centering
  % \includegraphics[width=0.49\linewidth]{Figures/Example1/errordecay_convex_1000.pdf}
  % \caption{$L^1$-error over interpolation nodes}
  % \label{fig: increasing-error decay}
% \end{minipage}
    \caption{Example $1$}
\end{figure}

For a given $\eta_t$ with $t$ selected interpolation nodes, we can approximate the error $\scrG(\eta_t)$ by discretizing the interval $[0,5]$ into $L=500$ discretization points $(y_1, \ldots, y_{L})$ and use the quadrature rule over these discretization points. 
% A single run consists of the computation of $30$ interpolation nodes. For each run, we can approximate the average error $\err(\eta_t)$ over time steps $t \in \{1, \ldots, 30\}$. To approximate this error, we discretize the $[0,5]$ into $500$ points $(y_1, \ldots, y_{500})$ and approximate the error $\err(\eta_t)$ with the quadratue rule over these discretization points.
To obtain a robust comparison, we run the algorithm $K=1000$ times and compute the average error, i.e., we compute
\begin{align}
\label{eq: expectation interpolation}
 E_t \coloneq \frac{1}{KL}\sum_{k=1}^{K}\sum_{i=1}^{L} J(y_i,\eta_t^k).
\end{align}
We keep adding points until $E_t$ is below a tolerance of $10^{-2}$. This means that we could end up with different configuration lengths for the different algorithms.
The decay in error in the number of time steps is shown in Figure~\ref{fig: increasing-error decay}. We first note that the error of R-PDM is lower than the error of uniform sampling. As a result, the selection procedure stops at $45$ nodes for R-PDM and only at $79$ for uniform sampling. We also ran the weak greedy algorithm (see \ref{eq: weak greedy}) with a sample size of $90$. We note that the error for this algorithm is slightly lower than R-PDM. The downside of the greedy algorithm is that the user has to a priori select a sample size, and it is unclear what this size should be. In this specific example, we cannot reasonably expect to reach the tolerance of $0.01$ for the weak greedy method if the sample size is smaller than $79$, i.e., the selected number of nodes in the uniform algorithm.
Moreover, for the weak greedy method, we evaluated $3510$ error estimates, while R-PDM only required $1034$ evaluations.

\begin{figure}[h]
    \begin{subfigure}{.49\textwidth}
    \centering
    \includegraphics[width=\linewidth]{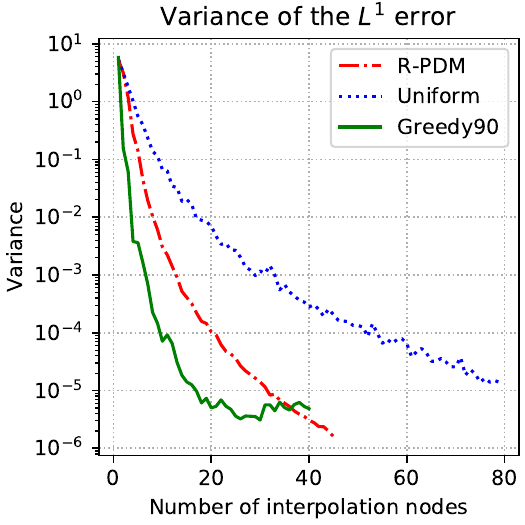}
    \caption{Variance of $L^1$-error per number of interpolation nodes of Example $1$}
    \label{fig: total-variance-increasing}
    \end{subfigure}
    \begin{subfigure}{.49\textwidth}
    \centering
    \includegraphics[width=\linewidth]{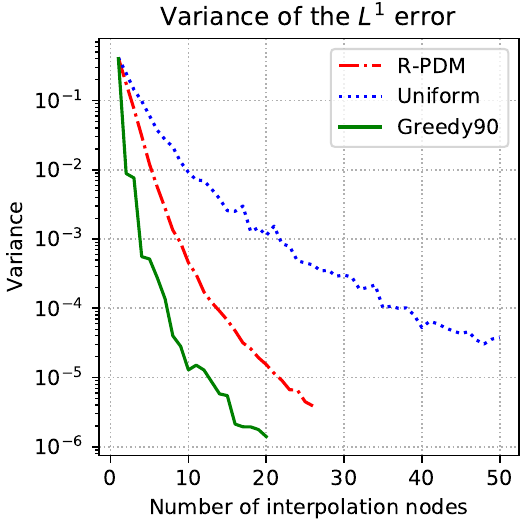}
    \caption{Variance of $L^1$-error per number of interpolation nodes of Example $2$}
    \label{fig: total-variance-convex}
    \end{subfigure}
    \caption{Comparing variances of $L^1$-error for Examples $1$ and $2$}
\end{figure}
% \begin{figure}[h]
%     \centering
%     \begin{minipage}[b]{0.49\linewidth}
%         \centering
%   \includegraphics[width=\linewidth]{Figures/Example1/variance_total_error_tol.pdf}
%   \caption{Variance of the $L^1$-error per number of interpolation points of Example $1$.}
%   \label{fig: total-variance-increasing}
%     \end{minipage}
%     \hfill 
%     \begin{minipage}[b]{0.49\linewidth}
%     \centering
%   \includegraphics[width=\linewidth]{Figures/Example2/variance_total_error_tol.pdf}
%   \caption{Variance of the $L^1$-error per number of interpolation points of Example $2$.}
%   \label{fig: total-variance-convex}
%     \end{minipage}
% \end{figure}

We also compare the variances of the two methods. First, we compute the variance of $\err(q,\eta_t)$ for time steps $t$, i.e., we compute
\begin{align}
\label{eq: variance interpolation}
V_t \coloneq \frac{1}{K} \sum_{k = 1}^{K} \left(\frac{1}{L} \sum_{i = 1}^{L} J(y_i,\eta_t^k)  - E_t\right)^2,
\end{align}
where $E_t$ is given by~\eqref{eq: expectation interpolation}.
This result is given in Figure~\ref{fig: total-variance-increasing}. At first, the variance is lower for the greedy method than for both R-PDM and uniform sampling, but for the final configuration, the variance of R-PDM is lower than the variance for weak greedy sampling.
%The variance decreases over time for all three methods, but the variance of R-PDM is by far the lowest. The variance for the final configuration is more than an order of magnitude lower than the variance for the weak greedy method. 
Despite the larger number of selected interpolation points, the variance of the final configuration achieved with uniform sampling is still greater than the variance for R-PDM. 

% The variance decreases over time for both methods, but especially for large numbers of selected interpolation points, the variance is over an order of magnitude smaller for R-PDM compared to uniform sampling. 
We also compute the average pointwise error $J(y_i,\eta_t)$ in the discretization points, $(y_1, \ldots, y_{L})$, for the final configuration $\eta_t$ (i.e., the configuration we find when reaching an $L^1$-error of $10^{-2}$). We again average this error over $K=1000$ runs, i.e., we compute
\begin{align}
\label{eq: num pointwise error interpolation}
    E_t(y_i) \coloneq \frac{1}{K}\sum_{k=1}^{K} J(y_i,\eta_t^k).
\end{align}

\begin{figure}[h]
    \begin{subfigure}{.49\textwidth}
    \centering
    \includegraphics[width=\linewidth]{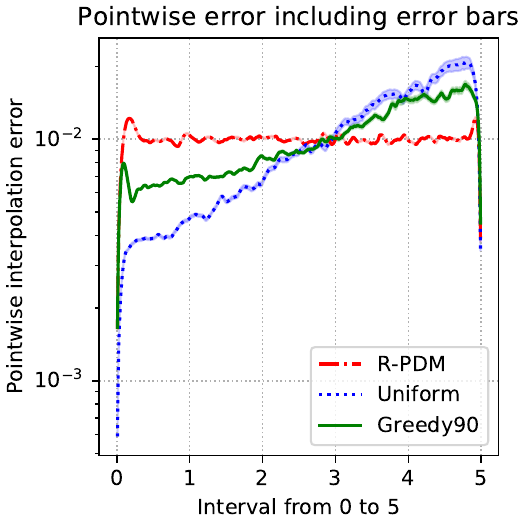}
    \caption{Error in discretization points with errorbars}
    \label{fig: increasing-average-error}
    \end{subfigure}
    \begin{subfigure}{.49\textwidth}
    \centering
    \includegraphics[width=\linewidth]{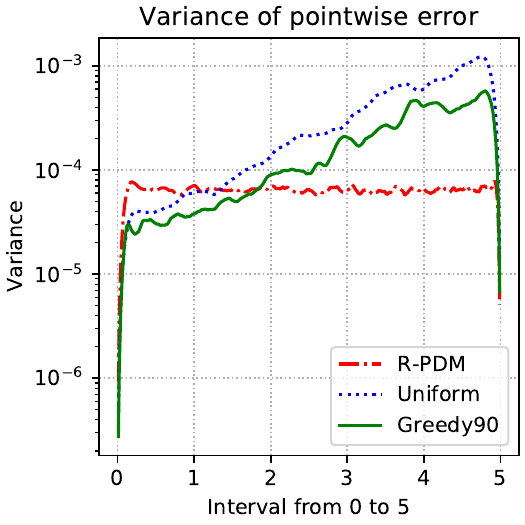}
    \caption{Variance in discretization points}
    \label{fig: increasing-variance}
    \end{subfigure}
    \caption{Pointwise error and variance of Example $1$}
\end{figure}

Figure~\ref{fig: increasing-average-error} displays these averaged pointwise errors, including error bars. These error bars are determined by the variance of both methods, i.e., we computed 
\begin{align}
\label{eq: pointwise variance interpolation}
V_t(y_i) \coloneq \frac{1}{K}\sum_{k=1}^K \left(J(y_i, \eta_t^k) - E_t(y_i)\right)^2.
\end{align}
For all methods, the error is low near the boundary, since the piecewise approximation is, by definition, always exact in the boundary points. We note that the pointwise error for R-PDM is always around $10^{-2}$ (except close to the boundary). Uniform sampling and the weak greedy method have lower pointwise errors on the left side of the interval and higher values on the right side. We note that overall, the error bars for R-PDM seem smaller than the error bars for uniform sampling and the weak greedy method. For clarity, the variance in the interval $[0,5]$ is shown in Figure~\ref{fig: increasing-variance}, and we note that the supremum variance of uniform sampling and the weak greedy method is an order of magnitude greater than the supremum variance of R-PDM. Moreover, the variance of R-PDM stays around $10^{-4}$, whereas the variance for the other two methods oscillates.

\subsubsection{Example 2}
In the second experiment, we consider the function (see Figure~\ref{fig: convex exact function})
\[
    f(x) = \frac{1}{200} \bigl((x-6)^4 + (x-2)^2 + 2\bigr),\qquad x\in[0,10].
\]
We note that this function is again infinitely many times differentiable and strongly convex, but this function is no longer increasing.
For R-PDM, we set $\alpha = 500$ and $\epsilon = 0.01$, and we again run both stochastic methods and the weak greedy method $1000$ times and select interpolation points until we reach a tolerance of $0.01$. We again discretize the interval $[0,10]$ into $500$ points $(y_1, \ldots, y_{L})$. The error $E_t$ given by~\eqref{eq: expectation interpolation} is displayed in Figure~\ref{fig: convex error decay}. We again note that the uniform sampling method selects way more interpolation nodes to reach the same error tolerance. We ran the weak greedy method with $90$ samples, and this again yields the lowest $L^1$-error. We note that the error of R-PDM is only slightly higher, but in the construction, R-PDM evaluated the pointwise error~\eqref{eq: error interpolation} $350$ times, whereas the weak greedy method required $1710$ evaluations. This difference in the number of evaluated samples is expected to increases in higher-dimensional problems (see \cite{nielen2025polytope}).

\begin{figure}[h]
    \begin{subfigure}{.49\textwidth}
    \centering
    \includegraphics[width=\linewidth]{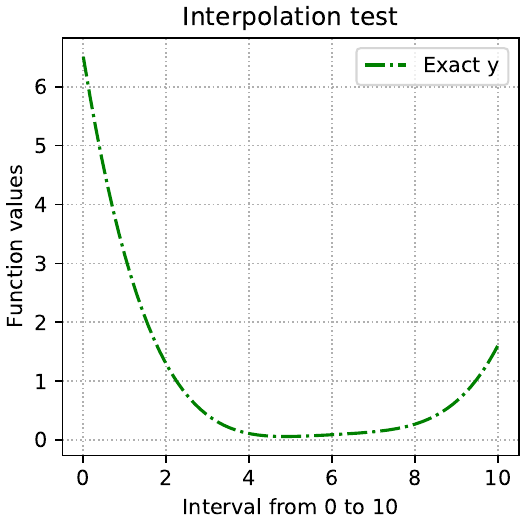}
    \caption{Exact function $f(x) = \frac{1}{200}((x-6)^4 + (x-2)^2 + 2)$}
    \label{fig: convex exact function}
    \end{subfigure}
    \begin{subfigure}{.49\textwidth}
    \centering
    \includegraphics[width=\linewidth]{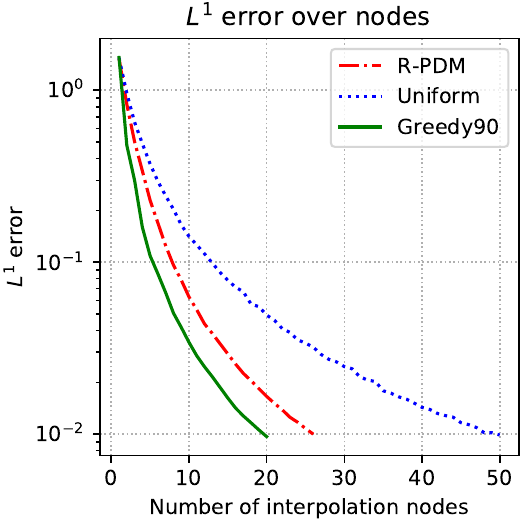}
    \caption{Decay of the $L^1$-error over the nodes}
    \label{fig: convex error decay}
    \end{subfigure}
    \caption{Example $2$}
\end{figure}

% We again note that this average error decreases for both methods, and the overall error is smaller for R-PDM.
The variance over time steps (see~\eqref{eq: variance interpolation}) is given in Figure~\ref{fig: total-variance-convex}. We note that this time the variance is lowest for the weak greedy method. Both the weak greedy method and R-PDM achieve a variance that is more than an order of magnitude less than the variance of uniform sampling in the final configuration, despite the larger number of selected interpolation nodes for uniform sampling.

% We again note that this variance is decreasing for both methods, but the variance for R-PDM is roughly two orders of magnitude lower than the variance for uniform sampling for $30$ selected interpolation nodes.
Lastly, the pointwise error~\eqref{eq: num pointwise error interpolation} in $500$ discretization points is displayed in Figure~\ref{fig: convex-average-error} with error bars. These error bars are again determined by the variance (for clarity displayed in Figure~\ref{fig: convex-variance}). 
% We note that the variance is low for both methods at the boundary. This is because the boundary points are, by definition, interpolation nodes. 
The errors of R-PDM and weak greedy sampling are comparable in this example. Compared to these two methods, the average error is lower for uniform sampling in the interval $[3,9]$ and higher everywhere else. This is roughly the same interval where the variance is lower for uniform sampling compared to R-PDM. From Figure~\ref{fig: convex exact function}, we note that this interval also corresponds to the flattest part of the function, and, therefore, the part that can be reasonably approximated with a linear function. Since R-PDM and weak greedy consider the approximation error in their selection procedure, fewer interpolation nodes are selected in this region compared to the uniform sampling. This explains the difference in variance. Apart from the boundary, the pointwise error variances for R-PDM and weak greedy sampling remain around the value $10^{-4}$. The variance for the uniform sampling shows greater variability. For uniform sampling, the variance can be very high in some regions and very low in others. 

% \begin{figure}[h]
%     \centering
%     \begin{minipage}[b]{0.49\linewidth}
%         \centering
%   \includegraphics[width=\linewidth]{Figures/interpolation_convex.pdf}
%   \caption{Exact function $f(x) = \frac{1}{200}((x-6)^4 + (x-2)^2 + 2)$}
%   \label{fig: convex exact function}
%     \end{minipage}
%     \hfill
%     \begin{minipage}[b]{0.49\linewidth}
%   \centering
%   \includegraphics[width=\linewidth]{Figures/Example2/errordecay_convex_1000.pdf}
%   \caption{Decay of the $L^1$ error over the interpolation nodes}
%   \label{fig: convex error decay}
% \end{minipage}
% \end{figure}

\begin{figure}[h]
    \begin{subfigure}{.49\textwidth}
    \centering
    \includegraphics[width=\linewidth]{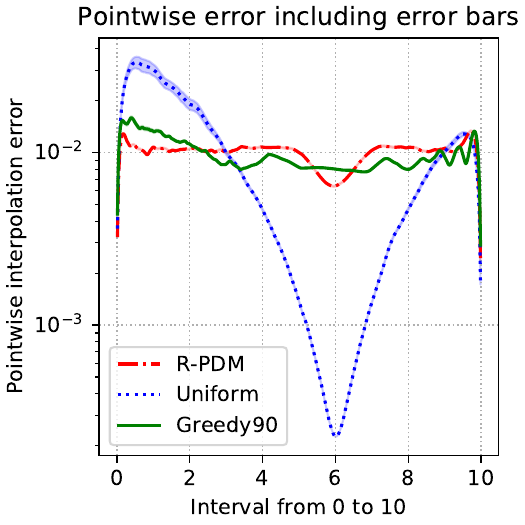}
    \caption{Error in discretization points with errorbars}
    \label{fig: convex-average-error}
    \end{subfigure}
    \begin{subfigure}{.49\textwidth}
    \centering
    \includegraphics[width=\linewidth]{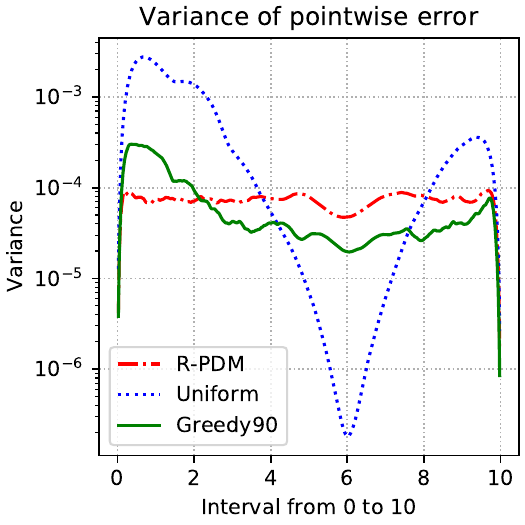}
    \caption{Variance in discretization points}
    \label{fig: convex-variance}
    \end{subfigure}
    \caption{Pointwise error and variance of Example $2$}
\end{figure}

\subsubsection{Example 3}
In this example, we consider the function $f(x) = \sin(2x)$ (see Figure~\ref{fig: sin-exact function}) in the interval $[0,10]$. We note that this function is still infinitely many times differentiable but no longer strongly convex. According to Theorem~\ref{theorem: interpolation C2}, we still expect convergence in this example.
%Therefore, this example does not satisfy Assumption~\ref{assumption: average error}, meaning Theorem~\ref{thm: convergence rate average} does not apply. Yet, we can still numerically observe convergence. 
For R-PDM, we set $\alpha = 500$ and $\eps = 0.01$, and we run every method $1000$ times and select interpolation points until we reach a tolerance of $0.01$. The interval $[0,10]$ is discretized into $500$ points $(y_1, \ldots, y_{L})$. 

\begin{figure}[h]
    \begin{subfigure}{.49\textwidth}
    \centering
    \includegraphics[width=\linewidth]{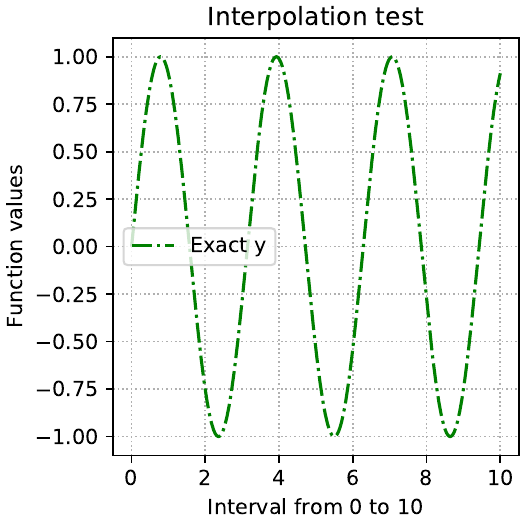}
    \caption{Exact function $f(x) = \sin(2x)$}
    \label{fig: sin-exact function}
    \end{subfigure}
    \begin{subfigure}{.49\textwidth}
    \centering
    \includegraphics[width=\linewidth]{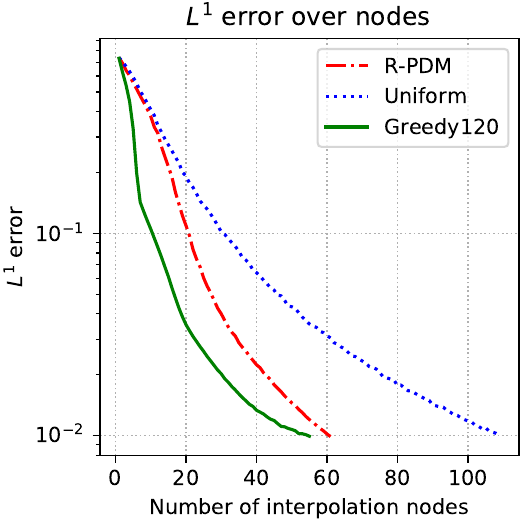}
    \caption{Decay of the $L^1$-error over the nodes}
    \label{fig: sin-error decay}
    \end{subfigure}
    \caption{Example $3$}
\end{figure}

\begin{figure}[h]
    \centering
    \includegraphics[width=0.49\linewidth]{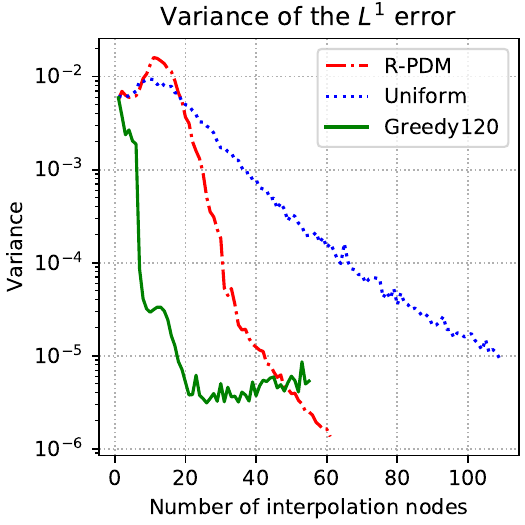}
    \caption{Variance of $L^1$-error per number of interpolation nodes of Example $3$}
    \label{fig: total-variance-sin}
\end{figure}

Figure~\ref{fig: sin-error decay} displays the average $L^1$-error given by~\eqref{eq: expectation interpolation}. We use $120$ samples in the weak greedy method. We again note that the $L^1$-error shows the slowest convergence for uniform sampling and the fastest for the weak greedy method, and as before, the difference between the weak greedy method and R-PDM is not large. We note that we evaluated $6480$ samples for the greedy method and only $1890$ samples for R-PDM. The variance of the $L^1$-error is given in Figure~\ref{fig: total-variance-sin}. We note that the variance of the weak greedy method is lowest at first, but starts oscillating at around $20$ selected interpolation nodes. For both R-PDM and uniform sampling, the variance increases at first, but for R-PDM the variance in the final configuration is lower than the variance for weak greedy.
%We note that for this example, the variance of the greedy method is clearly the lowest. For both R-PDM and uniform sampling, the variance increases at first. 
A possible explanation for the initial increase in variance for R-PDM and uniform sampling is that the $L^1$-error does not necessarily decrease when an interpolation node is added. In Figure~\ref{fig: sin-average-error}, we plot the pointwise error with error bars, and in Figure~\ref{fig: sin-variance}, we plot the pointwise variance. 
The pointwise error for uniform sampling and the weak greedy method are comparable, although the greedy method has a slightly lower variance. The pointwise error for R-PDM has lower peaks than the other two methods. We also note that the supremum of the pointwise variance is lower.
Theorem~\ref{theorem: interpolation C2} states the convergence of the global error function $\scrG$ defined in~\eqref{eq: error interpolation}. This convergence is numerically observed in the examples for R-PDM and uniform sampling. We even observe convergence for greedy sampling even though this method does not follow the same theoretical convergence results. Moreover, the numerical examples suggest directions for future research, since the convergence rates say nothing about the variance of the methods.
%The pointwise error is comparable for all three methods. The variance is slightly higher for uniform sampling and R-PDM. 
%Overall, we conclude that we can still observe convergence of the error, even though the assumptions in this paper are no longer satisfied. This indicates that in future research, the assumptions could potentially be relaxed.

% \begin{figure}[H]
%     \centering
%     \begin{minipage}[b]{0.49\linewidth}
%         \centering
%   \includegraphics[width=\linewidth]{Figures/Example3/sin.pdf}
%   \caption{Exact function $f(x) = \sin(2x)$}
%   \label{fig: sin-exact function}
%     \end{minipage}
%     \hfill
%     \begin{minipage}[b]{0.49\linewidth}
%   \centering
%   \includegraphics[width=\linewidth]{Figures/Example3/errordecay_convex_1000.pdf}
%   \caption{Decay of the $L^1$ error over the interpolation nodes}
%   \label{fig: sin-error decay}
% \end{minipage}
% \end{figure}

% \begin{figure}[H]
%     \centering
%     \begin{minipage}[b]{0.49\linewidth}
%         \centering
%   \includegraphics[width=\linewidth]{Figures/Example3/variance_total_error_tol.pdf}
%   \caption{Variance of the $L^1$ error per number of interpolation points of Example $3$}
%   \label{fig: total-variance-sin}
%     \end{minipage}
% \end{figure}

\begin{figure}[H]
    \begin{subfigure}{.49\textwidth}
    \centering
    \includegraphics[width=\linewidth]{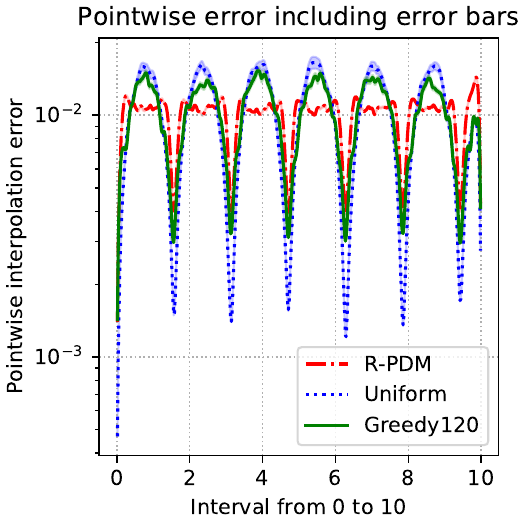}
    \caption{Error in discretization points with errorbars}
    \label{fig: sin-average-error}
    \end{subfigure}
    \begin{subfigure}{.49\textwidth}
    \centering
    \includegraphics[width=\linewidth]{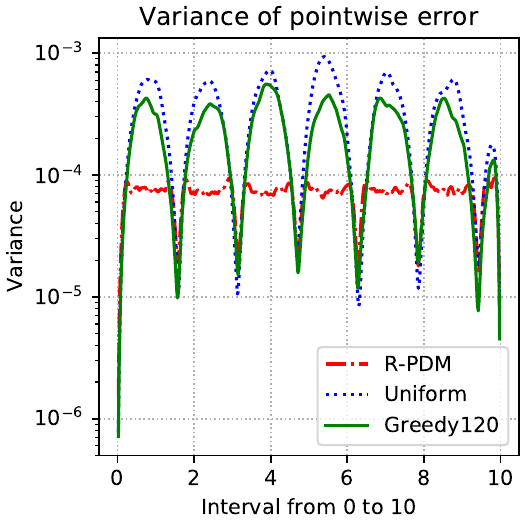}
    \caption{Variance in discretization points}
    \label{fig: sin-variance}
    \end{subfigure}
    \caption{Pointwise error and variance of Example $3$}
\end{figure}

\appendix
\section{Properties of \texorpdfstring{$\Omega$ and maps $\oplus$, $\sfN$ and $\mfS$}{TEXT}}
\label{app: compactness}
% This appendix contains the proof of Lemma~\ref{lemma: compactness omega}.
    
\begin{proof}[Proof of Lemma~\ref{lemma: compactness omega}]
    The compactness of $(\Omega, \mfd)$ follows from showing that it is complete and totally bounded.

    \medskip
    \emph{Completeness.} Let $(\eta^n)_{n \in \N}$ be a Cauchy sequence in $\Omega$. We need to show that $(\eta^n)_{n \in \N}$ converges and start by selecting a candidate limit. Let $0<\eps < \mathrm{diam}(P)$ and let $i \in \N$. Since $(\eta^n)_{n \in \N}$ is a Cauchy sequence, there exists an $N \in \N$, such that for all $n, m \ge N$ it holds that
        \begin{align*}
            \frac{1}{2^i} \bar{\mfd}(\eta_i^n, \eta_i^m) \le \mfd(\eta^n, \eta^m) <  \frac{\eps}{2^i},
        \end{align*}
        i.e., $(\eta_i^n)_{n \in \N}$ is Cauchy as well.
        We now claim that there exists an $M \in \N$, such that
        \begin{align}
        \label{claim M}
        \tag{\textasteriskcentered} 
        \begin{cases}
            \eta_i^n = \phi \quad &\text{for all}\, n \ge M,\; \text{or}\\
            \eta_i^n \in P \quad &\text{for all} \, n \ge M.
        \end{cases}
        \end{align}
        Indeed, for $0 < \eps < \rm{diam}(P)$, there exists an $M \in \N$ such that for all $n, m \ge M$
        \[\bar\mfd(\eta_i^n, \eta_i^m) < \eps.\]
        Therefore, \eqref{claim M} must hold since $\bar{\mfd}(\eta_i^n, \eta_i^m) = \rm{diam}(P) > \eps$ for any $\eta_i^n \in P$ and $\eta_i^m = \phi$. Since $\eta_i^n = \phi$, for all $n \ge M$ implies $\lim_{n\to\infty} \eta_i^n = \phi$, and $\eta_i^n \in P$ for all $n \ge M$ implies $(\eta_i^n)_{n \ge M}$ is a Cauchy sequence in $P \subset \R^d$, and $\R^d$ is complete, we conclude that $(\eta_i^n)_{n \in \N}$ has a limit.  
        % This follows from the definition of the metric $\mfd$, since $\bar{\mfd}(\eta_i^n, \eta_i^m) = \rm{diam}(P)$ if $\eta_i^n \in P$ and $\eta_i^m = \phi$, and therefore,
        % \begin{align*}
        %     \mfd(\eta^n, \eta^m) \ge \frac{1}{2^i} \rm{diam}(P).
        % \end{align*}
        % This means $(\eta_i^n)_{n \in \N}$ has a limit since the limit is $\phi$ in the first case, and in the second case, the limit exists since for $n$ large enough every $\eta_i^n \in P \subset \R^d$ and $\R^d$ is complete.

        We define the candidate limit of $(\eta^n)_{n \in \N}$ as $\eta^* = (\eta_1^*, \eta_2^*, \ldots)$, where $\eta_i^* \coloneq \lim_{n \to \infty} \eta_i^n$ for $i \in \N$. Then $\eta^* \in \Omega$. We next show that $\lim_{n\to \infty} \eta^n = \eta^*$.
        For all $k \in \N$, and for all $n, m \ge N$, 
        \begin{align*}
            \sum_{i = 1}^k \frac{1}{2^i} \bar{\mfd}(\eta_i^n, \eta_i^m)  < \eps.
        \end{align*}
        By passing $m$ to infinity, we find
        \begin{align*}
            \sum_{i = 1}^k \frac{1}{2^i} \bar{\mfd}(\eta_i^n, \eta_i^*)  < \eps.
        \end{align*}
        Then by passing $k$ to infinity, we get $\mfd(\eta^n, \eta^*) < \eps$, concluding the completeness proof.

        \medskip        
        \emph{Total boundedness.} Next, we show $(\Omega, \mfd)$ is totally bounded, i.e., for every $\eps > 0$, there exists a finite set of points $\Omega^\eps$, such that for every $\eta \in \Omega$, there exists a $\gamma \in \Omega^\eps$ such that $\mfd(\eta, \gamma) < \eps$.

        Let $\eps > 0$, then there exists a $K \in \N$, such that $\sum_{i = K+1}^\infty 2^{-i}\rm{diam}(P) < \eps/2$.
        Since $P \subset \R^d$ is compact, there exists a finite set $P^\eps$ such that for every $p \in P$, there exists an $x \in P^\eps$ such that $$|x-p| < \frac{\eps}{2K}.$$
        Then, let $\Omega^\eps$ be given by
        \begin{align*}
            \Omega^\eps \coloneq \bigl\{\gamma = (\gamma_1, \ldots, \gamma_K, \phi, \ldots) \in \Omega \, : \, \gamma_i \in P^\eps {\cup} \{\phi\}\bigr\}.
        \end{align*}
        Clearly, $\Omega^\eps$ is a finite set. Now, let $\eta \in \Omega$. If $\eta_i = \phi$ for $i \in \{1, \ldots, K\}$, set $\gamma_i = \phi$, otherwise set $\gamma_i = x$, with $x \in P^\eps$ such that $|\eta_i - x| < \eps/(2K)$. Moreover, let $\gamma_i = \phi$ for $i > K$. Then $\gamma \in \Omega^\eps$, and moreover
        \begin{align*}
            \mfd(\eta, \gamma) = \sum_{i=1}^K \frac{1}{2^i} \bar{\mfd}(\eta_i, \gamma_i) + \sum_{i=K+1}^\infty \frac{1}{2^i}\bar{\mfd}(\eta_i, \gamma_i) < K \frac{\eps}{2K} + \frac{\eps}{2} = \eps.
        \end{align*}
        Hence, we conclude $(\Omega, \mfd)$ is totally bounded, and therefore compact. 
        To show the continuity of $\sfN$, let $\eps > 0$ and $\eta \in \Omega_m$. Let $0< \delta <\min\{\eps, \text{diam}(P)/2^{m+1}\}$. Then for all $\gamma \in \Omega$, such that $\mfd(\eta, \gamma) < \delta$, it holds that
        $$\sfN(\gamma) - \sfN(\eta) = 0 < \eps.$$
        Hence, $\sfN$ is continuous and, therefore, measurable.

        Moreover, to show the continuity of $\oplus$, let $$\mfd_{\Omega \times P}((\eta, y), (\gamma, z)) \coloneq \mfd(\eta, \gamma) + |y - z|_2.$$
        Let $(\eta, y) \in \Omega \times P$ be chosen arbitrarily. Let $\eps > 0$, then for $0< \delta < \min\big(\eps, \rm{diam}(P)/2^{\sfN(\eta) + 1}\big)$, we have that for all $(\gamma, z) \in B_{\delta}((\eta, y))$, it holds that $\sfN(\eta) = \sfN(\gamma)$ since $\delta < \rm{diam}(P)/2^{\sfN(\eta) + 1}$. Therefore,
    \begin{align*}
    \mfd(\eta {\oplus} y, \gamma {\oplus} z) &=  \sum_{i=1}^{\sfN(\eta)} \frac{1}{2^i} \bar{\mfd}(\eta_i, \gamma_i) + \frac{1}{2^{\sfN(\eta) + 1}} |y-z|_2 \le \sum_{i=1}^{\sfN(\eta)} \frac{1}{2^i} \bar{\mfd}(\eta_i, \gamma_i) + |y-z|_2 < \eps. \qedhere
    \end{align*}
    \end{proof}

\begin{proof}[Proof of \ref{lemma: order map}]
   Let $\eps>0$ be arbitrary and $\eta \in \Omega$. Suppose $\mfd(\mfS(\gamma),\mfS(\eta)) <\eps$. We show that there exists some $\delta=\delta(\eps)>0$ such that
   \[
    \mfd(\gamma,\eta)<\delta\;\;\Longrightarrow\;\; \mfd(\mfS(\gamma),\mfS(\eta)) <\eps.
   \]
    We start by recalling from the proof of Lemma~\ref{lemma: compactness omega} that $0< \delta < \text{diam}(P)/2^{\sfN(\eta)+1}$, we have that $\sfN(\gamma)=\sfN(\eta)$. Moreover, we assume for the moment that \begin{equation}\label{property: distinct}\tag{$*$}
        \eta_i\ne \eta_j\quad \text{for $i\ne j$, $i,j=1,\ldots,\sfN(\eta)$}.
    \end{equation}

   For $\eta$ with property \eqref{property: distinct}, we define $d_{\min} \coloneq \min\{|\eta_i - \eta_j| \, | i \ne j, \, i, j \in \{1, \ldots, \sfN(\eta)\}\}> 0$. Now let $\sigma$ be the unique permutation such that
   \[
        \mfS(\eta) = (\eta_{\sigma(1)},\ldots, \eta_{\sigma(\sfN(\eta))},\phi,\ldots).
   \]
   We claim that for $\delta>0$ sufficiently small, we also have that
   \[
        \mfS(\gamma) = (\gamma_{\sigma(1)},\ldots, \gamma_{\sigma(\sfN(\eta))},\phi,\ldots).
   \]
   Indeed, choosing $\delta < \min\{\eps,d_{\min}\}/2^{\sfN(\eta)+1}$, we find that
   \[
        |\gamma_i - \eta_i| \le 2^{\sfN(\eta)-i} |\gamma_i - \eta_i| \le 2^{\sfN(\eta)}\mfd(\gamma,\eta) < \min\{\eps,d_{\min}\}/2\qquad\text{for all $i=1,\ldots,\sfN(\eta)$},
   \]
   thus implying that the $\gamma_i'$s are also distinct and that
   \[
        \gamma_{\sigma(i+1)} - \gamma_{\sigma(i)} \ge d_{\min} - |\gamma_{\sigma(i+1)} - \eta_{\sigma(i+1)}| - |\eta_{\sigma(i)} - \gamma_{\sigma(i)}| > 0 \qquad\text{for all $i=1,\ldots,\sfN(\eta)$},
   \]
   which verifies the claim. Consequently, we obtain
   \begin{align*}
        \mfd(\mfS(\gamma),\mfS(\eta)) = \sum_{i=1}^{\sfN(\eta)} \frac{1}{2^i} |\gamma_{\sigma(i)} - \eta_{\sigma(i)}| < \frac{1}{2} \min\{\eps,d_{\min}\} \le \eps,
    \end{align*}
    which proves the assertion under property \eqref{property: distinct}.

    For the general case, we choose $\tilde\eta\in\Omega$ with property \eqref{property: distinct} satisfying $\mfd(\tilde\eta,\eta) < \delta_1< \eps/2$ for some $\delta_1<0$ and for which there is a permutation $\sigma$ such that $\eta_{\sigma(i)}=\tilde\eta_{\sigma(i)}$ for all $i=1,\ldots,\sfN(\eta)$. In this way, we choose $\delta_2>0$ as before for $\mfd(\gamma,\tilde\eta)$ obtain
    \[
        \mfd(\mfS(\gamma),\mfS(\eta)) \le \mfd(\mfS(\gamma),\mfS(\tilde\eta)) + \mfd(\mfS(\tilde\eta),\mfS(\eta)) < \frac{\eps}{2} + \frac{\eps}{2} = \eps,
    \]
    thereby concluding the proof.
\end{proof}

\section{Well-posedness}
\label{app: well-definiteness}
We show that the process with generator $L$ given in~\eqref{eq: generator} exists. Before doing that we observe that the $L$ may be expressed as
\[
    LF(\eta) = \int_\Omega \bigl[F(\sigma) - F(\eta)\bigr]\,\kappa(\eta,d\sigma),
\]
where, for every $(\eta,A)\in\Omega\times\calB_\Omega$,
\[
    \kappa(\eta,A) \coloneq \int_P \delta_{\eta {\oplus} y}(A)\,\lambda(\eta,dy) = \int_P \mathbf{1}_A(\eta{\oplus}y)\,\lambda(\eta,dy).
\]
The idea is then to show that $\kappa$ is a well-defined bounded transition kernel. 

\begin{lemma}
    The map $\kappa{:}\Omega\times\calB_\Omega\to[0,+\infty)$ defined above is a bounded transition kernel.
\end{lemma}
\begin{proof}
    The boundedness of $\kappa$ follows directly from the boundedness of $\lambda$. Indeed, we have that $\kappa(\eta,\Omega) = \lambda(\eta,P)=1$ for every $\eta\in\Omega$.

    To show that $\kappa$ is a transition kernel, we start by noticing that the simple function $\Omega\times P\ni (\eta,y)\mapsto \mathbf{1}_A(\eta{\oplus}y)$ is Borel measurable for any $A\in\calB_\Omega$ since it is a composition of two Borel measurable maps. By a standard monotone class argument, we then have that $\Omega\ni \eta\mapsto \kappa(\eta,A)$ is Borel measurable.

    To show that $\kappa(\eta,\cdot)\in \calP(\Omega)$ for every $\eta\in\Omega$, we prove an equivalent definition of a measure. Clearly, $\kappa(\eta,\emptyset)=0$ and $\kappa(\eta,A{\cup} B) = \kappa(\eta,A) + \kappa(\eta,B)$ for disjoint sets $A,B\in\calB_\Omega$. Now let $\{A_i\}_{n\in\N}\subset\calB_\Omega$ be any increasing family of measurable sets such that $\cup_{i\in\N} A_i\in\calB_\Omega$. By the monotone convergence theorem and the continuity-from-below of the Dirac measure, we find
    \begin{align*}
        \lim_{i\to\infty} \kappa(\eta,A_i) &= \lim_{i\to\infty} \int_P \mathbf{1}_{A_i}(\eta{\oplus}y)\,\lambda(\eta,dy) = \int_P \lim_{i\to\infty} \mathbf{1}_{A_i}(\eta{\oplus}y)\,\lambda(\eta,dy) \\
        &= \int_P \lim_{i\to\infty} \delta_{\eta{\oplus}y}(A_i)\,\lambda(\eta,dy) = \int_P \delta_{\eta{\oplus}y}(\cup_{i\in\N} A_i)\,\lambda(\eta,dy) = \kappa(\eta, \cup_{i\in\N} A_i).
    \end{align*}
    Together with $\kappa(\eta,\Omega)=1$, this implies that $\kappa(\eta,\cdot)\in \calP(\Omega)$ for every $\eta\in\Omega$.
\end{proof}

\section*{Acknowlegdements}
This work is part of a project that has received funding from the European Research Council (ERC) under the
European Union’s Horizon 2020 Research and Innovation Programme (Grant Agreement No. 818473).
The authors thank Karen Veroy-Grepl for useful comments and discussions. The authors acknowledge using Grammarly and chatGPT to polish the written text for spelling, grammar, and general style.

\printbibliography
\end{document}